\documentclass[letterpaper]{amsart}

\usepackage[utf8]{inputenc}
\usepackage[T1]{fontenc}
\usepackage{lmodern}
\usepackage{amssymb}
\usepackage{mathtools}
\usepackage[pdfusetitle]{hyperref}
\usepackage{tikz-cd}

\def\arxiv#1{\href{http://arxiv.org/abs/#1}{\texttt{arXiv:#1}}}
\def\doi#1{\href{http://dx.doi.org/#1}{doi:#1}}

\mathtoolsset{mathic=true}
\let\intertext\shortintertext

\theoremstyle{plain}
\newtheorem{theorem}{Theorem}[section]
\newtheorem{proposition}[theorem]{Proposition}
\newtheorem{lemma}[theorem]{Lemma}
\newtheorem{corollary}[theorem]{Corollary}

\theoremstyle{definition}

\newtheorem{example}[theorem]{Example}
\newtheorem{remark}[theorem]{Remark}

\newtheorem{assumption}[theorem]{Assumption}

\theoremstyle{remark}
\newtheorem*{acknowledgements}{Acknowledgements}

\numberwithin{equation}{section}

\DeclareMathAlphabet\mathbfit{OML}{cmm}{b}{it}

\def\Z{\mathbb Z}
\def\Q{\mathbb Q}
\def\R{\mathbb R}

\DeclareMathOperator{\Hom}{Hom}
\DeclareMathOperator{\Ext}{Ext}

\DeclareMathOperator{\im}{im}
\DeclareMathOperator{\Char}{char}

\let\MFsm\setminus
\def\setminus{\mathbin{\mkern-2mu\MFsm\mkern-2mu}}

\def\UU{\mathcal U}
\def\FF{\mathcal F}
\def\OO{\mathcal O}
\def\II{\mathcal I}
\def\DD{\mathcal D}
\def\AA{\ell}

\def\AAv{\AA}

\def\kk{\Bbbk}
\def\kktilde{\skew{-2}\tilde\Bbbk}
\def\dirlim{\mathop{\underrightarrow\lim}}

\DeclareMathOperator{\depth}{depth}

\def\pair#1{{\langle#1\rangle}}
\def\Hc{H_{c}}
\def\HT{H_{T}}
\def\HTc{H_{T,c}}
\def\CT{C_{T}}
\def\CTc{C_{T,c}}

\def\barCTc{{\bar C}_{T,c}}
\def\barhCTc{{\bar C}^{T,c}}

\def\AB{AB}

\def\m{\mathfrak m}
\def\hHc{H^{c}}
\def\hCT{C^{T\!}}
\def\hCTc{C^{T,c}}
\def\hHT{H^{T\!}}
\def\hHTc{H^{T,c}}

\def\hatX{\widehat X}

\def\ie{\emph{i.\,e.}}
\def\cf{\emph{cf.}}

\def\EE{\mathcal{E}}
\def\II{I\mkern-3mu I}
\def\EI{{}^{I}\mkern-4mu E}
\def\EII{{}^{\II}\mkern-4mu E}
\def\dI{d^{\mkern 1mu I}}
\def\dII{d^{\mkern 1mu\II}}

\def\Cm{C_{\m}}
\def\Hm{H_{\m}}
\def\HmII{{}^{\II}\mkern-4mu H_{\m}}
\def\nn{\m_{L}}
\def\Cn{C_{\nn\!}}
\def\Hn{H_{\nn\!}}

\def\citeorbitsone#1{\cite[#1]{AlldayFranzPuppe}}

\begin{document}

\title[Equivariant Poincaré--Alexander--Lefschetz duality]{Equivariant Poincaré--Alexander--Lefschetz duality and the Cohen--Macaulay property}
\author{Christopher Allday}
\address{Department of Mathematics, University of Hawaii,
  2565~McCarthy Mall, Honolulu, HI~96822, U.S.A.}
\email{chris@math.hawaii.edu}
\author{Matthias Franz}
\address{Department of Mathematics, University of Western Ontario,
      London, Ont.\ N6A\;5B7, Canada}
\email{mfranz@uwo.ca}
\author{Volker Puppe}
\address{Fachbereich Mathematik, Universität Konstanz, 78457 Konstanz, Germany}
\email{volker.puppe@uni-konstanz.de}
\thanks{M.\,F.\ was partially supported by an NSERC Discovery Grant.}

\hypersetup{pdfauthor=\authors}
\hypersetup{pdftitle=Equivariant Poincaré-Alexander-Lefschetz duality and the Cohen-Macaulay property}


\begin{abstract}
  We prove a Poincaré--Alexander--Lefschetz duality theorem
  for rational torus-equivariant cohomology and rational homology manifolds.
  We allow non-compact and non-orientable spaces.
  We use this to deduce certain short exact sequences in equivariant cohomology,
  originally due to Duflot in the differentiable case, from similar, but more general short exact sequences in equivariant homology.
  A crucial role is played by the Cohen-Macaulayness of relative equivariant cohomology modules
  arising from the orbit filtration.
\end{abstract}

\maketitle

\section{Introduction}

Let \(T=(S^{1})^{r}\) be a torus, and let \(X\) be a \(T\)-space satisfying some
fairly mild assumptions (see Section~\ref{sec:4:assumptions}).
Recall that \(\HT^{*}(X)=\HT^{*}(X;\Q)\), the equivariant cohomology of~\(X\) with rational coefficients,
can be defined as the cohomology of the Borel construction (or homotopy quotient)~\(X_{T}=(ET\times X)/T\),
and that it is an algebra over the polynomial ring~\(R=H^{*}(BT)\).

In~\cite[p.~23]{Borel:1960}, A.~Borel observed that
``even if one is interested mainly in a statement involving only cohomology,
one has to use in the proof groups which play the role of homology groups,
and therefore this presupposes some homology theory''.
In this spirit we defined in~\cite{AlldayFranzPuppe}
the equivariant homology~\(\hHT_{*}(X)\)
of~\(X\), which is a module over~\(R\). 
In contrast to~\(\HT^{*}(X)\), it is not the homology of any space. Nevertheless, it has many
desirable properties: it is related to~\(\HT^{*}(X)\) via universal coefficient spectral sequences,
and, in the case of a rational Poincaré duality space~\(X\), also through an equivariant Poincaré duality isomorphism
\begin{equation}
  \label{eq:PD-iso-intro}
  \HT^{*}(X) \stackrel{\cong}\longrightarrow \hHT_{*}(X),
  \quad
  \alpha\mapsto \alpha\cap o_{T},
\end{equation}
which is the cap product with an equivariant orientation~\(o_{T}\in\hHT_{*}(X)\).
Note that unlike the non-equivariant situation, the isomorphism~\eqref{eq:PD-iso-intro}
does not necessarily translate into the perfection of the equivariant Poincaré pairing
\begin{equation}
  \label{eq:PD-pairing-intro}
  \HT^{*}(X) \times \HT^{*}(X) \to R,
  \quad
  (\alpha,\beta) \mapsto \pair{\alpha\cup\beta,o_{T}}.
\end{equation}
In fact, the pairing~\eqref{eq:PD-pairing-intro} is perfect if and only
if \(\HT^{*}(X)\) is a reflexive \(R\)-module, see~\cite[Cor.~1.3]{AlldayFranzPuppe}.
Hence, in the equivariant setting Poincaré duality
cannot be phrased in terms of cohomology alone.
Another reason to consider equivariant homology is that sometimes it behaves
better than cohomology. For example, the sequence
\begin{equation}
  \label{eq:exact-hHT-0-intro}
  0 \to \hHT_{*}(X^{T}) \to \hHT_{*}(X) \to \hHT_{*}(X,X^{T}) \to 0
\end{equation}
is always exact (see Proposition~\ref{thm:hHT-short-exact-intro} below), 
which is rarely the case for the corresponding sequence
in equivariant cohomology.

The first theme of the present paper is to extend Poincaré duality and its generalization
Poincaré--Alexander--Lefschetz duality to the torus-equivariant setting.
Equivariant Poincaré--Alexander--Lefschetz duality
for compact Lie groups and certain generalized (co)homology theories
has been discussed by Wirthmüller~\cite{Wirthmuller:1974}
and Lewis--May~\cite[\S III.6]{LewisMaySteinberger:1986},~\cite[\S XVI.9]{May:1996}
in the framework of equivariant stable homotopy theory.
Here we are interested in an explicit algebraic description
in the context of the singular Cartan model, \cf~Section~\ref{sec:4:singular-Cartan-model}.

We allow rational homology manifolds which may be non-compact or non-ori\-entable.
To this end we have to define equivariant cohomology with compact supports
and equivariant homology with closed supports, and, for non-orientable homology manifolds,
also equivariant (co)homology with twisted coefficients.

\begin{theorem}[Poincaré--Alexander--Lefschetz duality]
  \label{thm:PAL-intro}
  Let \(X\) be an orientable \(n\)-dimensional rational homology manifold
  with a \(T\)-action, and let \((A,B)\) be a closed \(T\)-stable pair in~\(X\).
  Then there is an isomorphism of \(R\)-modules 
  \begin{equation*}
     \HT^{*}(X\setminus B,X\setminus A) \cong \hHTc_{n-*}(A,B).
  \end{equation*}
\end{theorem}

Here \(\hHTc_{*}(A,B)\) denotes the equivariant homology of the pair~\((A,B)\)
with compact supports. 
Theorem~\ref{thm:PAL-intro} extends to an isomorphism of spectral sequences
induced by a \(T\)-stable filtration on~\(X\), and it implies
an equivariant Thom isomorphism.
We also prove analogous results for non-orientable manifolds
and twisted coefficients, which is essential for our applications.

Another important result in equivariant stable homotopy theory is the Adams isomorphism.
In Proposition~\ref{thm:locally-free-action} we prove a version of it in our context.

\smallskip

Our second theme is to extend the results of~\cite{AlldayFranzPuppe}
to the new (co)homology theories and to combine them with equivariant duality results.

Recall that the equivariant \(i\)-skeleton~\(X_{i}\subset X\) is the union of all \(T\)-orbits of dimension~\(\le i\);
this defines the orbit filtration of~\(X\).
A crucial observation, originally made by Atiyah~\cite{Atiyah:1974}
in the context of equivariant \(K\)-theory,
is that
the \(R\)-module \(\HT^{*}(X_{i},X_{i-1})\) 
is zero or Cohen--Macaulay of dimension~\(r-i\).
The same holds for equivariant homology, and it implies
the following result.

\begin{proposition}
  \label{thm:hHT-short-exact-intro}
  For any~\(0\le i\le r\) there is an exact sequence
  \begin{equation*}
    0 \to \hHT_{*}(X_{i}) \to \hHT_{*}(X) \to \hHT_{*}(X,X_{i}) \to 0.
  \end{equation*}
\end{proposition}

The case~\(i=0\) was made explicit in~\eqref{eq:exact-hHT-0-intro} above.
Again, this extends to homology with compact supports and/or twisted coefficients, see Proposition~\ref{thm:hHT-short-exact}.
Using the naturality properties of equivariant Poincaré--Alexander--Lefschetz duality, 
we can easily generalize a result of Duflot~\cite{Duflot:1983}
about smooth actions on differential manifolds, see Proposition~\ref{thm:duflot-general}: 

\begin{corollary}
  \label{thm:duflot-intro}
  Let \(X\) be a rational homology manifold.
  For any~\(0\le i\le r\) there is an exact sequence
  \begin{equation*}
    0 \to \HT^{*}(X,X\setminus X_{i}) \to \HT^{*}(X) \to \HT^{*}(X\setminus X_{i}) \to 0.
  \end{equation*}
\end{corollary}

\smallskip

We now turn to the relation between equivariant homology and the orbit filtration.
Recall that the \emph{Atiyah--Bredon complex}~\(\AB^{*}(X)\) is defined by
\begin{equation}
  \AB^{i}(X)=\HT^{*+i}(X_{i},X_{i-1})
\end{equation}
for~\(0\le i\le r\) and zero otherwise. (We set \(X_{-1}=\emptyset\).)
The differential
\begin{equation}
  d_{i}\colon \HT^{*}(X_{i},X_{i-1}) \to \HT^{*+1}(X_{i+1},X_{i})
\end{equation}
is the boundary map in the long exact sequence of the triple~\((X_{i+1},X_{i},X_{i-1})\).
In other words, \(\AB^{*}(X)\) is the \(E_{1}\)~page of the spectral sequence
arising from the orbit filtration and converging to~\(\HT^{*}(X)\),
and \(H^{*}(\AB^{*}(X))\) is its \(E_{2}\)~page.
A principal result of~\cite{AlldayFranzPuppe} is a natural isomorphism
\begin{equation}
  \label{eq:Ext-HAB-intro}
  H^{i}(\AB^{*}(X)) = \Ext_{R}^{i}(\hHT_{*}(X),R)
\end{equation}
for all~\(i\ge 0\). This is once again a consequence of the Cohen--Macaulay property of~\(\hHT_{*}(X_{i},X_{i-1})\).
In~\cite{AlldayFranzPuppe}
we used the isomorphism~\eqref{eq:Ext-HAB-intro} to study syzygies in equivariant cohomology
and to relate them to the Atiyah--Bredon complex.
Here we again indicate generalizations to (co)homology with the new pair of supports
and{\slash}or twisted coefficients.
They are used in~\cite{Franz:geocrit} to prove a ``geometric criterion'' for syzygies
in equivariant cohomology that only depends on the quotient~\(X/T\) as a stratified space.

\medskip

The paper is organized as follows.
In Section~\ref{sec:4:equiv-cohomology}
we first review equivariant cohomology with closed supports and equivariant homology with compact supports
and then define equivariant (co)homology with the other pair of supports.
We also consider homology manifolds and define variants of equivariant (co)homology
with twisted coefficients in this case.
Theorem~\ref{thm:PAL-intro} and its corollaries are proved in Section~\ref{sec:duality-results}.
Applications to the orbit structure are given in Section~\ref{sec:PAL-applications}.
There we also relate the cohomology of the Atiyah--Bredon complex
to the question of uniformity of an action.
Given the importance of~\eqref{eq:Ext-HAB-intro},
we include a direct proof of it in Section~\ref{sec:quick-proof}.
It uses only exact sequences as in Proposition~\ref{thm:hHT-short-exact-intro} 
and avoids the intricate reasoning with spectral sequences done in~\cite{AlldayFranzPuppe}.

\begin{acknowledgements}
  We thank an anonymous referee for numerous helpful comments and
  in particular for suggesting a strengthening of Proposition~\ref{thm:locally-free-action}.
\end{acknowledgements}

\section{Equivariant homology and cohomology}
\label{sec:4:equiv-cohomology}

\subsection{Notation and standing assumptions}
\label{sec:4:assumptions}

We write ``\(\subset\)'' for inclusion of sets and ``\(\subsetneq\)'' for proper inclusion.

Throughout this paper, \(T=(S^{1})^{r}\) denotes a compact torus of rank~\(r\ge0\),
and \(\kk\) a field. From Section~\ref{sec:PAL-applications} on we will assume
that the characteristic of~\(\kk\) is zero. All (co)homology is taken with coefficients in~\(\kk\)
unless specified otherwise. 

\(C_{*}(-)\) and \(C^{*}(-)\) denote normalized singular chains and cochains
with coefficients in the field~\(\kk \), and \(H_{*}(-)\)~and~\(H^{*}(-)\)
singular (co)ho\-mol\-ogy.
We adopt a cohomological grading, so that the homology
of a space lies in non-positive degrees; an element~\(c\in H_{i}(X)\)
has cohomological degree~\(-i\).

\(R=H^{*}(BT)\) is the symmetric algebra generated by~\(H^{2}(BT)\), and \(\m\lhd R\) its maximal homogeneous ideal.
All \(R\)-modules are assumed to be graded.
We consider \(\kk \) as an \(R\)-module (concentrated in degree~\(0\)) via the canonical augmentation.
For an \(R\)-module~\(M\) and an~\(l\in\Z\)
the notation~\(M[l]\) denotes a degree shift by~\(l\), so that the degree~\(i\)~piece of~\(M[l]\)
is the degree~\(i-l\)~piece of~\(M\). For the cohomology of some space,
we alternatively write \(H^{*}(X)[l]\) or \(H^{*-l}(X)\). Due to the cohomological grading,
we have in homology the identity~\(H_{*}(X)[l]=H_{*+l}(X)\).

We assume all spaces 
to be Hausdorff, second-countable, locally compact, locally contractible
and of finite covering dimension,
hence also completely regular, separable and metrizable.
Important examples are topological (in particular, smooth) manifolds, orbifolds, complex algebraic varieties,
and countable, locally finite CW~complexes.
We also assume that only finitely many distinct isotropy groups occur in any \(T\)-space~\(X\).

\begin{remark}
  \label{rem:quotient}
  Under these assumptions on a \(T\)-space~\(X\),
  the orbit space~\(X/T\) is again Hausdorff and locally compact \cite[Thm.~3.1]{Bredon:1972},
  second-countable,
  locally contractible \cite[Thm~3.8, Cor.~3.12]{Conner:1960}
  and of finite covering dimension \cite[Thm.~VIII.3.16]{Borel:1960}.
  It is easy to see that the same applies to the fixed point set \(X^{T}\)
  with the exception of local contractability: see Remark~\ref{rem:4:loc-contractible} below.
\end{remark}

It follows from our assumptions that every subset~\(A\subset X\) is paracompact, hence singular cohomology
and Alexander--Spanier cohomology are naturally isomorphic for all 
pairs~\((A,B)\)
such that \(A\) and~\(B\) are locally contractible.
We therefore put as another standing assumption
that all
subsets~\(A\subset X\) we consider
are locally contractible;
this holds automatically if \(A\) is open in~\(X\).
And we call \((A,B)\) a \(T\)-pair if \(A\) and \(B\) are \(T\)-stable.

In addition we will put a finiteness condition on the (co)homology
of the spaces and pairs we consider. This will be explained in detail
once we have defined equivariant (co)homology.

\subsection{The singular Cartan model}
\label{sec:4:singular-Cartan-model}

Let \(X\) be a \(T\)-space.
We recall from~\citeorbitsone{Sec.~\ref*{sec:equiv-cohomology}}
the definition of equivariant homology and cohomology via the ``singular Cartan model''.
As pointed out in~\cite{AlldayFranzPuppe},
it can be replaced by the usual Cartan model
for differentiable actions on manifolds and \(\kk=\R\).

The \emph{singular Cartan model} of the \(T\)-pair~\((A,B)\) in~\(X\) is
\begin{align}
  \label{eq:4:definition-CT}
  \CT^{*}(A,B) &= C^{*}(A,B)\otimes R
  \shortintertext{with \(R\)-linear differential}
  \label{eq:4:definition-d-CT}
  d(\gamma\otimes f) &=
  d\gamma\otimes f + \sum_{i=1}^{r}a_{i}\cdot\gamma\otimes t_{i}f \\
  \shortintertext{and \(R\)-bilinear product}
  (\gamma\otimes f)\cup(\gamma'\otimes f') &= \gamma\cup\gamma'\otimes f f'.
\end{align}
Here \(t_{1}\),~\ldots,~\(t_{r}\) are a basis of~\(H^{2}(BT)\subset R\), and
\(a_{1}\),~\ldots,~\(a_{r}\) are representative loops of the dual basis of~\(H_{1}(T)\);
the product~\(a_{i}\cdot\gamma\) refers to the action of~\(C_{*}(T)\) on~\(C^{*}(X)\)
induced by the \(T\)-action on~\(X\).
The equivariant chain complex~\(\hCT_{*}(A,B)\) 
is the \(R\)-dual of~\eqref{eq:4:definition-CT},
\begin{equation}
  \label{eq:4:definition-hCT}
  \hCT_{*}(A,B) = \Hom_{R}(\CT^{*}(A,B),R).
\end{equation}
Equivariant cohomology and homology are defined as
\begin{align}
  \HT^{*}(A,B) &= H^{*}(\CT^{*}(A,B)), \\
  \hHT_{*}(A,B) &= H_{*}(\hCT_{*}(A,B)).
\end{align}
This definition of~\(\HT^{*}(A.B)\) is naturally isomorphic,
as an \(R\)-algebra, to the usual one based on the Borel construction~\(X_{T}\).

\subsection{Other supports}

Let \((A,B)\) be a closed \(T\)-pair in a \(T\)-space~\(X\).
We define the \emph{equivariant cohomology of~\((A,B)\) with compact supports}
by
\begin{align}
  \HTc^{*}(A,B) &= \dirlim \HT^{*}(U,V) = H^{*}(\CTc^{*}(A,B)),
  \intertext{where}
  \label{eq:definition-CTc}
  \CTc^{*}(A,B) &= \dirlim \CT^{*}(U,V) = \bigl(\dirlim C^{*}(U,V)\bigr)\otimes R,
\end{align}
and the direct limits are taken over all \(T\)-stable open neighbourhood pairs~\((U,V)\) of~\((A,B)\)
such that \(X\setminus V\) is compact.
By tautness and excision, \(\HTc^{*}(A,B)\) is easily seen to be naturally isomorphic to the Alexander--Spanier cohomology
of the closure of~\((A,B)\) in the one-point compactification of~\(X\) relative to the added point.
Hence it does not matter whether \((A,B)\) is considered as a closed \(T\)-pair in~\(X\) or in~\(A\).

The \emph{equivariant homology of~\((A,B)\) with closed supports}
is defined by taking the \(R\)-dual
of~\eqref{eq:definition-CTc},
\begin{align}
  \hCTc_{*}(A,B) &= \Hom_{R}(\CTc^{*}(A,B),R), \\
  \hHTc_{*}(A,B) &= H_{*}(\hCTc_{*}(A,B)).
\end{align}
Clearly, we have \(\HTc^{*}(A,B)=\HT^{*}(A,B)\)
and \(\hHTc_{*}(A,B)=\hHT_{*}(A,B)\) if \(X\) is compact.

\subsection{Properties}
\label{sec:properties}

We list several important properties of equivariant (co)homology,
omitting proofs that were given in~\cite{AlldayFranzPuppe}.
In Section~\ref{sec:twisted} we will extend all results of this section
to homology manifolds and (co)homology with twisted coefficients,
see Remark~\ref{rem:twisted-properties}.

\begin{assumption}
  \label{ass:4:finite}
  For the rest of this paper we assume
  that \(H^{*}(A,B)\) is a finite-dimensional \(\kk\)-vector space
  for any \(T\)-pair~\((A,B)\) for which we consider equivariant
  cohomology with closed supports or equivariant homology with compact supports.
  By Proposition~\ref{thm:4:serre-ss} below, this implies
  that both \(\HT^{*}(A,B)\) and \(\hHT_{*}(A,B)\) are finitely generated \(R\)-modules.
  (Each of the latter conditions is actually equivalent to the former.)
\end{assumption}

\begin{proposition}[Serre spectral sequence]
  \label{thm:4:serre-ss}
  \label{thm:4:ss-HT-E2}
  \label{thm:4:ss-hHT-E2}
  Let \((A,B)\) be a \(T\)-pair in~\(X\). There are spectral sequences,
  natural in~\((A,B)\), with
  \begin{align*}
    E_{1} = E_{2} &= H^{*}(A,B)\otimes R  \;\Rightarrow\;  \HT^{*}(A,B), \\
    E_{1} = E_{2}&= H_{*}(A,B)\otimes R  \;\Rightarrow\;  \hHT_{*}(A,B).    
  \end{align*}
\end{proposition}

\begin{proof}
  These are eqs.~\eqref{eq:ss-HT-E2} and \eqref{eq:ss-hHT-E2}
  in~\cite{AlldayFranzPuppe}.
\end{proof}

\begin{proposition}[Universal coefficient theorem \citeorbitsone{Prop.~\ref*{thm:uct}}] 
  \label{thm:4:uct}
  Let \((A,B)\) be a \(T\)-pair in~\(X\). 
  There are spectral sequences,
  natural in~\((A,B)\), with
  \begin{align*}
    E_{2}^{p} &= \Ext_{R}^{p}(\HT^{*}(A,B),R) \;\Rightarrow\; \hHT_{*}(A,B), \\
    E_{2}^{p} &= \Ext_{R}^{p}(\hHT_{*}(A,B),R) \;\Rightarrow\; \HT^{*}(A,B).
  \end{align*}
\end{proposition}

For a multiplicative subset~\(S\subset R\) and a \(T\)-space~\(X\), define the \(T\)-stable subset
\begin{equation}
  X^{S} = \bigl\{\, x\in X \bigm| S\cap\ker(H^{*}(BT)\to H^{*}(BT_{x}))=\emptyset\,\bigr\} \subset X.
\end{equation}
It is closed in~\(X\), \cf~\cite[p.~132]{AlldayPuppe:1993}.
For example, \(X^{S}=X^{T}\) if \(\Char\kk=0\) and \(S\) contains all non-zero linear polynomials.

\begin{proposition}[Localization theorem]
  \label{thm:localization-thm-homology}
  Let \((A,B)\) be a \(T\)-pair in~\(X\), and let \(S\subset R\) be a multiplicative subset.
  Then the inclusion~\(X^{S}\hookrightarrow X\) induces isomorphisms
  of \(S^{-1}R\)-modules
  \begin{align*}
    S^{-1}\HT^{*}(A,B) &\to S^{-1}\HT^{*}(A^{S},B^{S}), \\
    S^{-1}\hHT_{*}(A^{S},B^{S}) &\to S^{-1}\hHT_{*}(A,B).
  \end{align*}
\end{proposition}

\begin{proof}
  The localization theorem for equivariant cohomology with closed supports
  is classical, \cf~\cite[Ch.~3]{AlldayPuppe:1993}.
  (Recall that only finitely many orbit types occur in~\(X\).)
  
  By the universal coefficient theorem
  there is a spectral sequence converging to~\(\hHT_{*}(A,B)\)
  with \(E_{2}\)~page~\(\Ext_{R}(\HT^{*}(A,B),R)\), and similarly
  for the pair~\((A^{S},B^{S})\).
  The inclusion~\(X^{S}\hookrightarrow X\) gives rise to a map of spectral sequences,
  which on the \(E_{2}\)~pages is the canonical map  
  \begin{equation}
    \label{eq:Ext-S}
    \Ext_{R}(\HT^{*}(A^{S},B^{S}),R) \to \Ext_{R}(\HT^{*}(A,B),R).
  \end{equation}
  Since localization is an exact functor, the \(S\)-localization of~\eqref{eq:Ext-S} is the map
  \begin{equation}
    \Ext_{S^{-1}R}(S^{-1}\HT^{*}(A^{S},B^{S}),S^{-1}R) \to \Ext_{S^{-1}R}(S^{-1}\HT^{*}(A,B),S^{-1}R),
  \end{equation}
  which is an isomorphism by the cohomological localization theorem.
  Hence, the localization of~\(\hHT_{*}(A^{S},B^{S}) \to \hHT_{*}(A,B)\)
  is an isomorphism as well.
\end{proof}

Let \(K\subset T\) be a subtorus, say of rank~\(p\), with quotient~\(L=T/K\).
In this case we have canonical morphisms of algebras
\begin{equation}
  \label{eq:def-RK-RL}
  H^{*}(BL)=R_{L}=\kk[t_{p+1},\dots,t_{r}]\to R\to H^{*}(BK)=R_{K}=\kk[t_{1},\dots,t_{p}].
\end{equation}
Moreover, any choice of splitting~\(T\cong K\times L\) defines an isomorphism~\(R=R_{K}\otimes R_{L}\).

\begin{proposition}
  \label{thm:action-trivial}
  Let \(K\subset T\) be a subtorus with quotient~\(L=T/K\).
  Let \((A,B)\) be a closed \(T\)-pair in~\(X\)
  such that \(K\) acts trivially on~\(A\setminus B\).
  Then there are isomorphisms of \(R\)-modules
  \begin{align*}
    \HT^{*}(A,B) &= H_{L}^{*}(A,B)\otimes_{R_{L}} R, \\
    \hHT_{*}(A,B) &= H^{L}_{*}(A,B)\otimes_{R_{L}} R.
  \end{align*}
  The result holds for any \(T\)-pair~\((A,B)\) if \(K\) acts trivially on all of~\(A\).
\end{proposition}

In the proof below as well as in that of Proposition~\ref{thm:locally-free-action}
we will use the following fact, \cf~\cite[Cor.~B.1.13]{AlldayPuppe:1993}:
Let \(\phi\colon M\to N\) be a quasi-isomorphism of dg~\(R\)-modules.
If \(M\)~and~\(N\) are free as \(R\)-modules,
then \(\phi\) is a homotopy equivalence over~\(R\).

\begin{proof}
  Since \(B\) is closed in~\(A\), we have,
  by tautness and excision, a quasi-iso\-mor\-phism of dg \(R\)-modules
  \begin{equation}
    \label{eq:CTAB-dirlim-quiso}
    \CT^{*}(A,B) \to \dirlim \CT^{*}(A\setminus B,U\setminus B) = \Bigl(\dirlim C^{*}(A\setminus B,U\setminus B)\Bigr)\otimes R,
  \end{equation}
  where the direct limit is taken over all \(T\)-stable open sets~\(U\supset B\).
  Hence we may work with this direct limit, which we denote by~\(M\).
  
  Now choose a splitting~\(T=K\times L\).
  By~\citeorbitsone{Prop.~\ref*{thm:hHT-independent-basis}}, we may assume
  that the representatives~\(a_{i}\in C_{1}(T)\) appearing in the ``Cartan differential''~\eqref{eq:4:definition-d-CT}
  are chosen such that \(a_{1}\),~\ldots,~\(a_{p}\in C_{1}(K)\) and \(a_{p+1}\),~\ldots,~\(a_{r}\in C_{1}(L)\).
  Since we are using normalized singular (co)chains,
  \(C_{*}(K)\) acts trivially on each~\(C^{*}(A\setminus B,U\setminus B)\).
  The differential on~\(M\) therefore takes the form
  \begin{equation}
    d(\gamma\otimes f) = d\gamma\otimes f +\sum_{i=p+1}^{r} a_{i}\cdot\gamma\otimes t_{i}f,
  \end{equation}
  which implies
  \begin{equation}
    \HT^{*}(A,B)=H_{L}^{*}(A,B)\otimes H^{*}(BK)
    = H_{L}^{*}(A,B)\otimes_{R_{L}} R
  \end{equation}
  by the Künneth formula.

  By the remark made above, the quasi-isomorphism~\eqref{eq:CTAB-dirlim-quiso}
  is a homotopy equivalence, which is preserved by the functor~\(\Hom_{R}(-,R)\).
  For equivariant homology we can therefore argue analogously.

  The last claim follows by the five-lemma from the previous one, applied to~\(A\) and~\(B\) separately,
  and the long exact sequence of the pair.
\end{proof}

At the other extreme, we have the following:

\begin{proposition}
  \label{thm:locally-free-action}
  Let \(K\subset T\) be a subtorus, say of rank~\(p\), with quotient~\(L=T/K\).
  Let \((A,B)\) be closed a \(T\)-pair in~\(X\)
  such that \(K\) acts freely on~\(A\setminus B\) (or just locally freely in case~\(\Char\kk=0\)).
  Then \(H^{*}(A/K,B/K)\) is finite-dimensional, and there are isomorphisms of
  \(R_{L}\)-modules
  \begin{align*}
    \HT^{*}(A,B) &= H_{L}^{*}(A/K,B/K), \\
    \hHT_{*}(A,B) &= H^{L}_{*-p}(A/K,B/K).
  \end{align*}
  The result holds for any \(T\)-pair~\((A,B)\) if \(K\) acts (locally) freely on all of~\(A\).
\end{proposition}

The cohomological part is well-known, cf~\cite[Prop.~3.10.9]{AlldayPuppe:1993}.
That a degree shift by~\(-p\) is necessary for the homological part can already
be seen by considering~\(K=T=X\):
In this case one has \(\hHT_{*}(X)=\kk[-r]=H_{*-r}(X/T)\), \cf~\citeorbitsone{Ex.~\ref*{ex:homogeneous-space}}.

Geometrically, the homological isomorphism
can be understood as a transfer 
for the quotient map~\(X\to X/K\).
Since in the singular setting it is delicate to define a transfer map or integration over the fibre
on the (co)chain level, we will follow an algebraic approach
and postpone the geometrical aspects to our discussion of Poincaré--Alexander--Lefschetz duality
(Remark~\ref{rem:locally-free}).
The homology isomorphism can also be viewed
as a version of the Adams isomorphism in equivariant stable homotopy theory
(see~\cite[\S II.7]{LewisMaySteinberger:1986} or~\cite[\S XVI.5]{May:1996})
in our algebraic context.

The proof of Proposition~\ref{thm:locally-free-action} requires some preparation.
Recall that \(\m=(t_{1},\dots,t_{r})\) is the maximal graded ideal in~\(R\).
In the proof below we will use the local duality isomorphism
\begin{equation}
  \label{eq:4:local-duality}
  \Hm^{j}(M) = \Ext_{R}^{r-j}(M,R[2r])^{\vee},
\end{equation}
which is natural in the \(R\)-module~\(M\),
see for instance \cite[Thm.~A1.9]{Eisenbud:2005}
(where the generators of the polynomial ring are assigned the degree~\(1\), not~\(2\)).
The symbol~``\({}^{\vee}\)'' in~\eqref{eq:4:local-duality} denotes the dual of a graded \(\kk\)-vector space.
We will also need the Čech complex 
computing \(\Hm^{*}(M)\)
by means of 
some generators
of~\(\m\) as in~\cite[p.~189]{Eisenbud:2005}.

More generally, we consider the Čech complex
for a dg~\(R\)-module~\(M\).
Thus we obtain a bicomplex~\(\Cm^{*,*}(M)\) with
first differential~\(\dI\) coming from~\(M\)
and second differential~\(\dII\)
coming from the Čech complex for the canonical generators~\(t_{1}\),~\ldots,~\(t_{r}\).
An element in~\(\Cm^{i,j}(M)\) is a sum of elements of degree~\(j\)
in the \(i\)-fold localizations in this Čech complex.
While \(j\) is unbounded in both directions,
we have \(0\le i\le r\), so that
both filtrations of the bicomplex are regular \cite[p.~452]{Bredon:1997}.
Hence both associated spectral sequences converge to~\(\Hm^{*}(M)\),
the cohomology of~\(\Cm^{*,*}(M)\) with respect to the total differential.

In the first bicomplex spectral sequence we have
\begin{align}
  \EI_{1} &= \Cm^{*,*}(H^{*}(M)), \\ 
  \label{eq:ss2}
  \EI_{2} &= \Hm^{*}(H^{*}(M)) 
\end{align}
since the cohomology of the localization of~\(M\) is the localization of the cohomology.

Taking the other bicomplex spectral sequence, we get
\begin{equation}
  \label{eq:ss1}
  \EII_{1} = \HmII^{*}(M) 
\end{equation}
where \(\HmII^{*}(M)\) means the cohomology of~\(\Cm^{*,*}(M)\) with respect to the differential~\(\dII\),
that is, the local cohomology of the \(R\)-module~\(M\) with trivial differential.

Suppose that \(M\) is finitely generated and free as an \(R\)-module.
By local duality one then has that the \(E_{1}\)~page
\begin{equation}
  \EII_{1}^{k} = \HmII^{k}(M) =
  \begin{cases}
    \Hom_{R}(M,R[2r])^{\vee} & \text{if \(k=r\),} \\
    0 & \text{otherwise}
  \end{cases}
\end{equation}
is concentrated in the column~\(k=r\), and therefore
\begin{equation}
  \Hm^{*}(M) = H^{*}(\Hom_{R}(M,R[2r]))^{\vee}[r] = H^{*}(\Hom_{R}(M,R))^{\vee}[-r].
\end{equation}

If \(M\) is \(R\)-homotopy equivalent to some~\(M'\), then so are \(\Cm^{*,*}(M)\) and \(\Cm^{*,*}(M')\),
hence \(\Hm^{*}(M)\cong\Hm^{*}(M')\) as \(R\)-modules. In particular, if \(M=\CT^{*}(X)\), then it is
\(R\)-homotopy equivalent to a dg~\(R\)-module~\(M'\) which is finitely generated and free as an \(R\)-module.
We therefore conclude that
\begin{align}
  \Hm^{*}(\CT^{*}(X)) &= H^{*}(\Hom_{R}(\CT^{*}(X),R))^{\vee}[-r] \\
  &= H^{*}(\hCT_{*}(X))^{\vee}[-r] = \hHT_{*}(X)^{\vee}[-r].
\end{align}

Let \(\nn=(t_{p+1},\dots,t_{r})\) be the maximal graded ideal of~\(R_{L}\).
Using the canonical generators,
we can similarly define \(\Cn^{*,*}(-)\) and \(\Hn^{*}(-)\)
for dg~\(R_{L}\)-modules, hence a~fortiori for dg~\(R\)-modules.
Since these generators are among the chosen generators of~\(\m\),
we have a canonical map of bicomplexes
\begin{equation}
  \Cm^{*,*}(M) \to \Cn^{*,*}(M)
\end{equation}
for any dg~\(R\)-module~\(M\), inducing a map of \(R\)-modules~\(\Hm^{*}(M)\to\Hn^{*}(M)\).

\begin{lemma}
  \label{thm:Hn-Hm-iso}
  Let \((A,B)\) be a closed \(T\)-pair in~\(X\).
  Assume that \(K\) acts freely on~\(A\setminus B\), and that all~\(x\in A\setminus B\) have the same isotropy group, say \(K'\).
  Then the map~\(\Hm^{*}(\HT^{*}(A,B))\to\Hn^{*}(\HT^{*}(A,B))\) is an isomorphism.

  If \(\Char\kk=0\), then it is enough that \(K\) acts locally freely and that the isotropy groups in~\(A\setminus B\)
  have the same identity component~\(K'\).
\end{lemma}

\begin{proof}
  Since \(K\) acts (locally) freely, the composition \(K'\to T\to L\) is injective (or has finite kernel).
  This implies that the composition~\(H^{*}(BL)\to H^{*}(BT)\to H^{*}(BK')\) is surjective.
  Hence there are \(t_{1}'\),~\ldots,~\(t_{p}'\in\nn\)
  such that \(t_{i}\) and \(t_{i}'\) map to the same element in~\(H^{*}(BK')\)
  for~\(1\le i\le p\), and \(u_{i}=t_{i}-t_{i}'\) maps to~\(0\). By the localization theorem (Proposition~\ref{thm:localization-thm-homology}),
  this implies that the localization of~\(\HT^{*}(A,B)\) at~\(u_{i}\) vanishes.

  We observe that \(u_{1}\),~\ldots,~\(u_{p}\),~\(t_{p+1}\),~\ldots,~\(t_{r}\) also generate \(\m\).
  Since local cohomology can be computed from any set of generators, \cf~\cite[Thm.~A1.3]{Eisenbud:2005}, 
  we can assume that one has chosen these generators
  instead of the canonical generators~\(t_{1}\),~\ldots,~\(t_{r}\).
  Then the terms in the Čech complex involving at least one of the~\(u_{i}\)'s
  drop out, and we are left with the Čech complex computing~\(\Hn^{*}(\HT^{*}(A,B))\).
\end{proof}

\begin{proof}[Proof of Proposition~\ref{thm:locally-free-action}]
  We choose a splitting~\(T=K\times L\) with compatibly chosen representatives \(a_{i}\in C_{1}(T)\)
  as in the proof of Proposition~\ref{thm:action-trivial}.

  As mentioned already, the isomorphism
  \begin{equation}
    \label{eq:iso-HT-HL-quotient-K}
    \HT^{*}(A,B) = H_{L}^{*}(A/K,B/K)
  \end{equation}
  is classical
  (and requires that \(B\) is closed in~\(A\)).
  It is induced by the quasi-isomorphism of dg~\(R_{L}\)-modules
  \begin{align}
    \label{eq:quasi-iso-XK}
    C_{L}^{*}(A/K,B/K) = C^{*}(A/K,B/K)\otimes R_{L} &\to \CT^{*}(A,B) = C^{*}(A,B)\otimes R, \\
    \notag
    \gamma\otimes f &\mapsto \pi^{*}\gamma\otimes f,
  \end{align}
  where \(\pi\colon X\to X/K\) is the projection.

  It follows from the localization theorem that the localization of~\(H_{K}^{*}(A,B)\)
  at each generator~\(t_{i}\) of~\(R_{K}\) vanishes. Since \(H_{K}^{*}(A,B)\)
  is finitely generated over~\(R_{K}\) by Assumption~\ref{ass:4:finite},
  this implies that \(H_{K}^{*}(A,B)\) is killed by some power of each~\(t_{i}\)
  and therefore that it is finite-dimensional as \(\kk\)-vector space. By taking \(T=K\)
  in~\eqref{eq:iso-HT-HL-quotient-K}, we see that \(H^{*}(A/K,B/K)\)
  is also finite-dimensional.

  For the homological statement
  we start by proving that the canonical map
  \begin{equation}
    \Cm^{*,*}(\CT^{*}(A,B)) \to \Cn^{*,*}(\CT^{*}(A,B))
  \end{equation}
  is a quasi-isomorphism.
  We proceed by induction on the number~\(m\) of (connected) orbit types in~\(A\setminus B\).
  For~\(m=0\) there is nothing to show as \(A=B\) in this case.
  Otherwise fix an orbit type of maximal dimension in~\(A\setminus B\)
  and let \(A'\subset A\) be the union of~\(B\) and all other orbit types;
  \(A'\) is \(T\)-stable and closed in~\(A\).
  
  The short exact sequence
  \begin{equation}
    0 \to \CT^{*}(A,A') \to \CT^{*}(A,B) \to \CT^{*}(A',B) \to 0
  \end{equation}
  gives rise to the commutative diagram
  \begin{equation}
  \begin{tikzcd}[column sep=small]
    0 \arrow{r} & \Cm^{*,*}(\CT^{*}(A,A')) \arrow{d} \arrow{r} & \Cm^{*,*}(\CT^{*}(A,B)) \arrow{d} \arrow{r} & \Cm^{*,*}(\CT^{*}(A',B)) \arrow{d} \arrow{r} & 0 \\
    0 \arrow{r} & \Cn^{*,*}(\CT^{*}(A,A')) \arrow{r} & \Cn^{*,*}(\CT^{*}(A,B)) \arrow{r} & \Cn^{*,*}(\CT^{*}(A',B)) \arrow{r} & 0 \mathrlap{,}
  \end{tikzcd}
  \end{equation}
  whose horizontal sequences are again short exact.
  The right vertical arrow is a quasi-isomorphism by induction.
  To see that the left one is so as well, we consider the induced map
  between the \(E_{2}\)~pages of the first bicomplex spectral sequences~\eqref{eq:ss2}.
  In our case this is the map
  \begin{equation}
    \Hm^{*}(\HT^{*}(A,A')) \to \Hn^{*}(\HT^{*}(A,A')),
  \end{equation}
  and it is an isomorphism by Lemma~\ref{thm:Hn-Hm-iso} and our choice of~\(A'\).
  The map induced in cohomology by the left arrow above
  therefore is also an isomorphism.
  Hence the middle arrow is a quasi-isomorphism by the five-lemma, which proves the claim.

  The quasi-isomorphism~\eqref{eq:quasi-iso-XK}
  is in fact a homotopy equivalence over~\(R_{L}\)
  as both sides are free as \(R_{L}\)-modules.
  We therefore get isomorphisms of \(R_{L}\)-modules
  \begin{align}
    \hHT_{*}(A,B)^{\vee}[-r] &= \Hm^{*}(\CT^{*}(A,B)) = \Hn^{*}(\CT^{*}(A,B)) \\
    &= \Hn^{*}(C_{L}^{*}(A/K,B/K))
    = H^{L}_{*}(A/K,B/K)^{\vee}[-(r-p)],
  \end{align}
  which translates into the claimed isomorphism
  \begin{equation}
    \hHT_{*}(A,B) \to H^{L}_{*}(A/K,B/K)[-p] = H^{L}_{*-p}(A/K,B/K).   
  \end{equation}
  
  The last claim follows again from the absolute case and the five-lemma.
\end{proof}

All these results hold as well
for cohomology with compact supports and homology with closed supports
and closed \(T\)-pairs~\((A,B)\),
assuming that \(\Hc^{*}(A,B)\) is a finite-dimensional \(\kk\)-vector space, \cf~Assumption~\ref{ass:4:finite}.
The proofs are identical; the localization theorem for cohomology
with compact supports follows from the version for closed supports
since direct limits preserve isomorphisms.

\begin{proposition}
  \label{thm:relative-cohomology-complement}
  For any closed \(T\)-pair~\((A,B)\) in~\(X\) there are isomorphisms of \(R\)-modules
  \begin{align*}
    \HTc^{*}(A,B) &= \HTc^{*}(A\setminus B), \\
    \hHTc_{*}(A,B) &= \hHTc_{*}(A\setminus B).
  \end{align*}
\end{proposition}

\begin{proof}
  The first identity follows from excision and the fact that a direct limit
  is an exact functor. The second identity then is a consequence
  of the universal coefficient theorem.
\end{proof}

\subsection{Homology manifolds}

Let \(X\) be a \emph{\(\kk\)-homology manifold}, say of dimension~\(n\).
By this we mean a connected space~\(X\) such that for any~\(x\in X\) one has
\begin{equation}
  \label{eq:def-homology-mf}
  H_{i}(X,X\setminus\{x\}) \cong
  \begin{cases}
    \kk & \text{if \(i=n\),} \\
    0 & \text{if \(i\ne n\).}
  \end{cases}
\end{equation}
If in addition \(\hHc_{n}(X) \cong \kk\), then
\(X\) is called \emph{orientable}.
Homology manifolds are an appropriate setting for Poincaré duality,
see Lemma~\ref{thm:PD-nonequiv} below.

\begin{assumption}
  \label{assumption:pm1}
  We assume that any homology manifold~\(X\) we consider
  is orientable or
  admits an orientable two-fold covering~\(\pi\colon \tilde X\to X\).
\end{assumption}

For non-orientable~\(X\), such a covering will be called
an \emph{orientation cover}. Note that \(\tilde X\) is necessarily connected.
For orientable~\(X\) we define the trivial two-fold covering
to be the orientation cover.
We will use orientation covers to define (co)homology with twisted coefficients
in Section~\ref{sec:twisted}.

\begin{remark}
  Any \(\Z\)-homology manifold admits an orientation cover, but
  it seems unclear whether this holds for arbitrary
  \(\kk\)-homology manifolds, see the discussion in~\cite[p.~331]{Bredon:1997}.
  On the other hand, if an orientation cover exists, then it is unique.
  For orientable~\(X\), this is true by definition.
  For non-orientable~\(X\) it can be seen as follows:
  
  Let \(\gamma\) be a loop at~\(x\in X\).
  By transporting local orientations along~\(\gamma\),
  we get an automorphism of~\(H_{n}(X,X\setminus\{x\})\), \cf~\cite[p.~39]{Bredon:1960},
  which is necessarily multiplication by some non-zero scalar.
  This induces a morphism~\(\phi\colon\pi_{1}(X)\to\kk^{\times}\).
  The connected orientable covers of~\(X\) are of the form~\(\tilde X/G\)
  where \(\tilde X\) is the universal cover
  and \(G\) a subgroup of the kernel of~\(\phi\).
  In particular,
  there is at most one orientation cover.
\end{remark}

The following observation seems to be well-known, but we could not find a suitable reference.

\begin{lemma}
  Assume \(\Char\kk=0\). Any connected, locally orientable orbifold
  is a \(\kk\)-homology manifold satisfying Assumption~\ref{assumption:pm1}.
\end{lemma}

See~ \cite[Sec.~1.1]{AdemLeidaRuan:2007} or~\cite{Satake:1956}
for the definition of an orbifold.
By ``locally orientable'' we mean that
locally the orbifold~\(X\), say of dimension~\(n\), is the quotient of an open ball in~\(\R^{n}\)
by a finite subgroup of~\(SO(n)\).

\begin{proof}
  Condition~\eqref{eq:def-homology-mf} holds
  because one locally divides by a finite subgroup of~\(SO(n)\)
  and \(\Char\kk=0\).

  The existence of an orientation cover can be shown in the same way as for manifolds.
  Recall that in the smooth case one proceeds as follows, \cf~\cite[Ch.~15--17]{Lee:2012}:
  If \(X\) admits an oriented atlas, that is, if the charts of~\(X\) can be oriented in a way consistent with coordinate changes,
  then one can integrate differential forms with compact supports,
  and the integration map provides an isomorphism \(\hHc_{n}(X)\cong\kk\).
  Otherwise \(\hHc_{n}(X)=0\), and one can construct a connected double cover with oriented atlas
  by doubling all charts and gluing them according
  to whether coordinate changes preserve or reverse chart orientations.
  Hence \(X\) is orientable in our sense if and only if it
  admits an oriented atlas.

  For an orbifold~\(X\) one can also define differential forms with compact supports,
  and if \(X\) is locally orientable and has an atlas of compatibly oriented
  charts, then one can integrate these forms, \cf~\cite[\S 8]{Satake:1956}. If such an atlas does not
  exist, then one can again pass to an oriented two-fold cover.
  Now the proofs for manifolds go through without change. 
\end{proof}

\begin{lemma}
  Let \(X\) be a \(\kk\)-homology manifold with orientation cover~\(\pi\colon\tilde X\to X\).
  Any \(T\)-action on~\(X\) lifts to a \(T\)-action on~\(\tilde X\).
\end{lemma}

See~\cite[Cor.~I.9.4]{Bredon:1972} for an analogous result
in the context of topological manifolds.

\begin{proof}
  The case of orientable~\(X\) is trivial. If \(X\) is non-orientable,
  then by~\cite[Sec.~I.9]{Bredon:1972} the \(T\)-action on~\(X\) lifts
  to a \(\tilde T\)-action on~\(\tilde X\),
  where \(\tilde T\) is a two-fold covering of~\(T\)
  and \(\ker(\tilde T \to T)\cong\Z_{2}\) acts by deck transformations.
  If the non-trivial deck transformation~\(\tau\)
  were orientation-preserving, then \(X\) would have to be orientable
  because \(\hHc_{n}(X)=\hHc_{n}(\tilde X)^{\tau}\cong\kk\),
  where \(n=\dim X\).
  So \(\tau\) does not preserve orientations, which implies
  that \(\tilde T\) cannot be connected. Hence its identity component is \(T\),
  and the action lifts.
\end{proof}

\subsection{Twisted coefficients}
\label{sec:twisted}

The aim of this section is to introduce equivariant (co)homology with twisted coefficients~\(\kktilde\).
To distinguish it from the (co)homology we have considered so far,
the latter will be called (co)homology with constant coefficients~\(\kk\) from now on.
Twisted coefficients are only interesting if the characteristic of the ground field~\(\kk\) differs from~\(2\), which we assume
in this section.
For~\(\Char\kk=2\) (co)homology with twisted coefficients is defined to be the same as (co)homology with constant coefficients.

We focus on cohomology with closed supports and homology with compact supports. 
All results are equally valid for the other pair of supports;
we will indicate when proofs for that case need additional arguments.

\smallskip

Let \(X\) be a \(\kk\)-homology manifold (which, by our definition, is connected)
with orientation cover~\(\pi\colon\tilde X\to X\)
and non-trivial deck transformation~\(\tau\).
For a pair~\((A,B)\) in~\(X\),
we write \((\tilde A,\tilde B)=(\pi^{-1}(A),\pi^{-1}(B))\).
Moreover, we denote the involution of~\(C^{*}(\tilde A,\tilde B)\)
induced
by~\(\tau\)
by the same letter.
Since \(2 \in \kk \) is invertible, we get a decomposition
\begin{equation}
  \label{eq:splitting-cochains-1}
  C^{*}(\tilde A,\tilde B) = C^{*}(\tilde A,\tilde B)_{+}\oplus C^{*}(\tilde A,\tilde B)_{-}
\end{equation}
into the eigenspaces of~\(\tau\) for the eigenvalues~\(\pm 1\).
Note that \(\pi^{*}\) is an isomorphism of~\(C^{*}(A,B)\) onto~\(C^{*}(\tilde A,\tilde B)_{+}\).
We define \(C^{*}(A,B;\kktilde)\), the cochains on~\((A,B)\) with twisted coefficients,
to be the eigenspace for the eigenvalue~\(-1\) of~\(\tau\).
Hence the splitting~\eqref{eq:splitting-cochains-1} becomes
\begin{equation}
  \label{eq:splitting-cochains}
  C^{*}(\tilde A,\tilde B) = C^{*}(A,B)\oplus C^{*}(A,B;\kktilde)
\end{equation}
and induces an analogous decomposition in cohomology,
\begin{equation}
  \label{eq:splitting-cohomology}
  H^{*}(\tilde A,\tilde B) = H^{*}(A,B)\oplus H^{*}(A,B;\kktilde).
\end{equation}


We now assume that \(X\) is equipped with a \(T\)-action and that the pair~\((A,B)\) is \(T\)-stable. Since
the decomposition~\eqref{eq:splitting-cochains} is \(C_{*}(T)\)-stable,
we can 
define \emph{equivariant (co)homology with twisted coefficients}
in a way analogous to Section~\ref{sec:4:singular-Cartan-model}:
\begin{align}
  \CT^{*}(A,B;\kktilde) &= \CT^{*}(\tilde A,\tilde B)_{-} = C^{*}(A,B;\kktilde)\otimes R \\
  \intertext{with the same differential as in~\eqref{eq:4:definition-d-CT},}
  \label{eq:def-HT-kktilde}
  \HT^{*}(A,B;\kktilde) &= H^{*}(\CT^{*}(A,B;\kktilde)), \\
  \hCT_{*}(A,B;\kktilde) &= \Hom_{R}(\CT^{*}(A,B;\kktilde),R), \\
  \label{eq:def-hHT-kktilde}
  \hHT_{*}(A,B;\kktilde) &= H_{*}(\hCT_{*}(A,B;\kktilde)).
\end{align}
Note that 
one has decompositions
\begin{align}
  \label{eq:splitting-HT}
  \HT^{*}(\tilde A,\tilde B) &= \HT^{*}(A,B)\oplus \HT^{*}(A,B;\kktilde), \\
  \label{eq:splitting-hHT}
  \hHT_{*}(\tilde A,\tilde B) &= \hHT_{*}(A,B)\oplus \hHT_{*}(A,B;\kktilde).
\end{align}
(For \(\HTc^{*}(-)\) and \(\hHTc_{*}(-)\)
they follow from the fact that sets of the form~\(\tilde V\)
such that the complement of~\(V\subset X\) is compact are cofinal among all
subsets of~\(\tilde X\) with compact complement.)
Of course, one already has decompositions on the (co)chain level.

Assumption~\ref{ass:4:finite} is extended as follows:

\begin{assumption}
  \label{ass:4:finite-homologymf}
  For any \(T\)-pair~\((A,B)\) in a \(\kk\)-homology manifold~\(X\)
  and any (co)homology theory we are going to consider,
  we assume that the non-equivariant cohomology of the cover~\((\tilde A,\tilde B)\) is finite-dimensional over~\(\kk\).
  In light of~\eqref{eq:splitting-cohomology}, this is equivalent to both the cohomology
  with constant coefficients and that with twisted coefficients being finite-dimensional.
  By Proposition~\ref{thm:4:serre-ss},
  this in turn implies that the equivariant (co)homology of~\((A,B)\) with constant or twisted coefficients
  is finitely generated over~\(R\).
\end{assumption}

Our definition of~\(H^{*}(A,B;\kktilde)\) does not require \(A\) or \(B\) to be \(\kk\)-homology manifolds themselves.
But if \(A\) is connected and open in~\(X\), then it is a \(\kk\)-homology manifold as well, and the restriction of~\(\pi\) to~\(A\) is
the orientation cover of~\(A\). Hence the definition of twisted coefficients
is independent of the ambient space in this case.

\begin{remark}
  \label{rem:twisted-properties}
  All results from Section~\ref{sec:properties}
  (Serre spectral sequences, universal coefficient theorems and localization theorems)
  carry over to twisted coefficients.
  To see this, one can either redo the proofs with twisted (co)homology,
  or one can reduce the new results to the untwisted case
  by using the splittings~\eqref{eq:splitting-HT} and~\eqref{eq:splitting-hHT}.
\end{remark}

\begin{remark}
  An alternative way to define cohomology with twisted coefficients
  is to use local coefficient systems. This could be done as well in the equivariant setting,
  and one could even dispense with Assumption~\ref{assumption:pm1}.
  The drawback of this approach would be that one cannot reduce statements
  to the case of constant coefficients anymore.
  In particular, one would need to prove a generalization
  of Proposition~\ref{thm:locally-free-action}
  (essentially, of the Vietoris--Begle mapping theorem)
  to local coefficients, which is required to prove Proposition~\ref{thm:stratum-cm-c}.
\end{remark}

\begin{remark}
  \label{rem:4:loc-contractible}
  We are mainly interested in applying our results to the fixed point sets~\(X^{K}\)
  of subtori~\(K\subset T\), and
  because we want to use the Localization Theorem for singular cohomology
  (Proposition~\ref{thm:localization-thm-homology}),
  we put local contractability
  into the standing assumptions in Section~\ref{sec:4:assumptions}
  
  Now it is a small step from~\eqref{eq:definition-CTc}
  to using Alexander--Spanier cohomology
  for all closed invariant pairs~\((A,B)\),
  \cf~\citeorbitsone{Rem.~\ref*{rem:loc-contractible}}.
  Thus it is not, in fact, necessary to assume
  closed subsets to be locally contractible
  since the Localization Theorem for Alexander--Spanier cohomology
  does not need this assumption. We would, however,
  continue to assume that the ambient space~\(X\)
  satisfies the standing assumptions; and we do not know
  of any torus action on such a space where the fixed point sets
  are not locally contractible -- but nor do we know a proof
  that they always are.
\end{remark}

\section{Equivariant duality results}
\label{sec:duality-results}

\subsection{Poincaré duality}

Let \(\kk\) be a field.
We start with the statement of non-equivariant Poincaré duality
for orientable homology manifolds in our setting because it is not easy
to locate it in the literature in the desired generality.

\begin{lemma}
  \label{thm:PD-nonequiv}
  Let \(X\) be an orientable \(\kk\)-homology manifold of dimension~\(n\).
  For any non-zero~\(o\in\hHc_{n}(X)\),
  the cap product map
  \begin{equation*}
    \Hc^{*}(X) \to H_{n-*}(X),
    \quad
    \alpha\mapsto \alpha\cap o
  \end{equation*}
  is an isomorphism. 
\end{lemma}

Such an~\(o\) is called an orientation of~\(X\);
it generalizes the notion of a fundamental class of a manifold.

\begin{proof}
  Recall that \(X\) is assumed to be a locally compact and locally contractible second-countable Hausdorff space. 
  Sheaf (co)homology and singular (co)homology (with closed or compact supports)
  are therefore naturally isomorphic on~\(X\),
  \cf~\cite[Thm.~III.1.1, Cor.~V.12.17, Cor.~V.12.21]{Bredon:1997}.

  As mentioned in the proof of~\cite[Cor.~V.16.9]{Bredon:1997},
  the stalks of the orientation sheaf on~\(X\) are given by~\eqref{eq:def-homology-mf}.
  Hence \(X\) is an \(n\)-dimensional homology manifold over~\(\kk\)
  in the sense of~\cite[Def.~V.9.1]{Bredon:1997}.
  Moreover, by~\cite[Thm.~V.16.16\,(f)]{Bredon:1997}
  our definition of orientability coincides with the one in~\cite[Def.~V.9.1]{Bredon:1997}.
  By~\cite[Thm.~V.9.2, Cor.~V.10.2]{Bredon:1997},
  the sheaf-theoretically defined cap product with~\(o\)
  is an isomorphism.
  This map coincides with the
  cap product in the singular theory given above,
  \cf~\cite[Ex.~V.22]{Bredon:1997}.
\end{proof}

Now let \(X\) be a \(T\)-space and
a not necessarily orientable \(\kk\)-homology manifold of dimension~\(n\) with orientation cover~\(\tilde X\).
The cup product in~\(\tilde X\) is \(\tau\)-equiv\-ari\-ant,
so that we obtain a pairing
\begin{equation}
  \label{eq:cup-product-twisted-c}
  \CTc^{*}(X;\kktilde) \otimes \CT^{*}(X)\to \CTc^{*}(X;\kktilde),
\end{equation}
hence a cap product
\begin{equation}
  \label{eq:cap-product-twisted-c}
  \HTc^{*}(X;\kktilde) \otimes \hHTc_{*}(X;\kktilde) \to \hHT_{*}(X),
  \quad
  \alpha\otimes b\mapsto \alpha\cap b.
\end{equation}
Extending the above definition,
an \emph{orientation} of~\(X\)
is a non-zero element~\(o\in \hHc_{n}(X;\kktilde)\subset\hHc_{n}(\tilde X)\).
An \emph{equivariant orientation} is an element~\(o_{T}\in\hHTc_{n}(X;\kktilde)\)
that restricts to an orientation under the restriction map~\(\hHTc_{*}(X;\kktilde)\to\hHc_{*}(X;\kktilde)\).

\begin{proposition}
  \label{thm:PD-noncompact}
  Let \(X\) be an \(n\)-dimensional \(\kk \)-homology manifold. 
  Any orientation~\(o\) of~\(X\) lifts uniquely to an equivariant orientation~\(o_{T}\).
  Moreover, taking the cap product with~\(o_{T}\) gives an isomorphism of \(R\)-modules
  \begin{equation*}
    \HTc^{*}(X;\kktilde) \stackrel{\cap o_{T}}\longrightarrow \hHT_{n-*}(X)
  \end{equation*}
  and, dually, an isomorphism~
  \begin{equation*}
    \hHTc_{*}(X;\kktilde)\longrightarrow\HT^{n-*}(X).
  \end{equation*}
\end{proposition}

\begin{proof}
  The canonical projection~\(\hHc_{*}(X;\kktilde)\otimes R\to\hHc_{*}(X;\kktilde)\)
  is the edge homomorphism of the \(E_{2}\)~page of the Serre spectral sequence
  for~\(\hHTc_{*}(X;\kktilde)\) 
  (Proposition~\ref{thm:4:serre-ss} and Remark~\ref{rem:twisted-properties}).
  Since \(\hHc_{*}(X)\otimes R\) lives
  in homological degrees at most~\(n\), there are no higher differentials,
  and the map~\(\hHTc_{n}(X;\kktilde)\to\hHc_{n}(X;\kktilde)\)
  is an isomorphism. Hence any orientation lifts uniquely
  to an equivariant orientation.

  To prove the first isomorphism, let us assume for the moment that \(X\) is orientable.
  Applying the Serre spectral sequence to the map
  \begin{equation}
    \label{eq:eq-duality-proof}
    \HTc^{*}(X) \to \hHT_{*}(X),
    \quad
    \alpha \mapsto \alpha\cap o_{T},
  \end{equation}
  we find on the \(E_{2}\)~level the \(R\)-linear extension
  \begin{equation}
    \Hc^{*}(X)\otimes R \to H_{*}(X)\otimes R,
    \quad
    \alpha\otimes f \mapsto \alpha\cap o\otimes f
  \end{equation}
  of the non-equivariant Poincaré duality isomorphism
  from Lemma~\ref{thm:PD-nonequiv},
  which is therefore an isomorphism, too.
  The non-orientable case reduces to the orientable one:
  Since \(\hHc_{n}(\tilde X)=\hHc_{n}(\tilde X)_{-}=\hHc_{n}(X;\kktilde)\),
  capping with~\(o_{T}=\tilde o_{T}\)
  restricts to the claimed isomorphism. 
  
  The second isomorphism is a consequence of the first and the universal coefficient theorem
  (Proposition~\ref{thm:4:uct} and Remark~\ref{rem:twisted-properties}).
\end{proof}

\begin{remark}
  If \(X\) is orientable, then the two eigenspaces of~\(\tau\) in the decomposition
  \begin{equation}
    \HT^{*}(\tilde X) = \HT^{*}(X) \oplus \HT^{*}(X;\kktilde)
  \end{equation}
  are isomorphic as \(R\)-modules and even as modules over~\(\HT^{*}(X)\).
  Hence (co)ho\-mol\-ogy with twisted coefficients and with constant coefficients
  are isomorphic in this case, and these isomorphisms are compatible
  with the two Poincaré duality isomorphisms
  \(\HTc^{*}(X)\to \hHT_{*}(X)\) and \(\HTc^{*}(X;\kktilde)\to \hHT_{*}(X)\).
  Of course, in the case~\(\Char\kk=2\) there is no difference either.
\end{remark}

\subsection{Poincaré--Alexander--Lefschetz duality}
\label{sec:pal-duality}

Classically, Poincaré--Alexander--Lefschetz duality (also called ``Poincaré--Lefschetz duality'')
refers to an isomorphism of vector spaces 
\begin{equation}
  \label{eq:PAL-classical}
  \Hc^{*}(A,B) \cong H_{n-*}(X\setminus B,X\setminus A)
\end{equation}
for any closed pair~\((A,B)\) in an oriented \(n\)-dimensional manifold~\(X\),
\cf~\cite[\S VIII.7]{Dold:1980} 
or~\cite[\S VI.8]{Bredon:1993},
for instance.
In this section we generalize this to equivariant (co)ho\-mol\-ogy, and
we also derive a spectral sequence version of it.
Our approach is similar to the one in~\cite{Bredon:1993}.
We continue to assume that \(X\) is an \(n\)-dimensional 
\(\kk \)-homology manifold with a \(T\)-action.

\goodbreak
  
Our first result
includes Theorem~\ref{thm:PAL-intro} from the introduction.

\begin{theorem}[Poincaré--Alexander--Lefschetz duality]
  \label{thm:PAL-A-B}
  Let \((A,B)\) be a closed \(T\)-pair in~\(X\).
  Then there is a
  commutative diagram
  \begin{equation*}
    \begin{tikzcd}[font=\small,column sep=small]
    \rar & \HTc^{n-*}(A,B;\kktilde) \dar \rar & \HTc^{n-*}(A;\kktilde) \dar \rar & \HTc^{n-*}(B;\kktilde) \dar \rar{d} & \HTc^{n+1-*}(B;\kktilde) \dar \rar & {} \\
    \rar & \hHT_{*}(X\setminus B,X\setminus A) \rar & \hHT_{*}(X,X\setminus A) \rar & \hHT_{*}(X,X\setminus B) \rar{d} & \hHT_{*-1}(X,X\setminus B) \rar & {}
  \end{tikzcd}
  \end{equation*}
  all of whose vertical arrows are isomorphisms. 
  An analogous diagram exists with the roles of homology and cohomology interchanged
  and all arrows reversed. 
\end{theorem}


Since we also want to prove an extension to spectral sequences,
we place ourselves in a slightly more general situation.
Consider an increasing filtration
\begin{equation}
  \emptyset=X_{-1}\subset X_{0}\subset\dots\subset X_{m}=X
\end{equation}
of~\(X\) by 
closed \(T\)-stable subsets.
(
See Remark~\ref{rem:4:loc-contractible} on how to do without the standing assumption
of local contractability.)
We set \(\hatX_{i}=X\setminus X_{i}\), so that the decreasing filtration
of~\(X\) by the open complements of the~\(X_{i}\) can be written as
\begin{equation}
  X=\hatX_{-1}\supset \dots \supset \hatX_{m-1} \supset \hatX_{m}=\emptyset.
\end{equation}

\def\CU#1{C^{*}(#1\,|\,\UU;\kktilde)}
\def\CTU#1{\CT^{*}(#1\,|\,\UU;\kktilde)}
Let \(\pi\colon\tilde X\to X\) be the orientation cover, and let
\(U=(U_{-1},U_{0},\dots,U_{m})\) be an increasing sequence of open subsets of~\(X\)
such that \(X\setminus U_{-1}\) is compact and \(X_{i}\subset U_{i}\) for all~\(i\).
Any such sequence determines an open cover
\begin{equation}
  \UU=\bigl\{\pi^{-1}(U_{0}),\pi^{-1}(U_{1}\setminus X_{0}),\dots,\pi^{-1}(U_{m}\setminus X_{m-1})\bigr\}
\end{equation}
of~\(\tilde X\). We write \(\CU{X}\) for the complex of \(\UU\)-small cochains
and similarly \(\CTU{X}=\CU{X}\otimes R\) for the corresponding singular Cartan model.
We start by establishing a variant of the cup product~\eqref{eq:cup-product-twisted-c}.

\begin{lemma}
  \label{thm:relative-cup-Ui}
  For any~\(-1\le i\le j\le m\)
  there is a well-defined relative cup~product
  \begin{equation*}
    \CT^{*}(X,U_{j};\kktilde)\otimes\CT^{*}(\hatX_{i}) \stackrel{\cup}\longrightarrow \CTU{X}.
  \end{equation*}
  It is compatible with restrictions in the sense that the diagram
  \begin{equation*}  
  \begin{tikzcd}
    \CT^{*}(X,U_{j};\kktilde)\otimes\CT^{*}(\hatX_{j}) \arrow{r}{\cup} & \CTU{X} \\
    \CT^{*}(X,U_{j};\kktilde)\otimes\CT^{*}(\hatX_{i}) \arrow{r}{\cup} \arrow{u} \arrow{d} & \CTU{X}  \arrow{u}[right]{=} \arrow{d}[right]{=} \\
    \CT^{*}(X,U_{i};\kktilde)\otimes\CT^{*}(\hatX_{i}) \arrow{r}{\cup} & \CTU{X}
  \end{tikzcd}
  \end{equation*}
  commutes.
\end{lemma}

\begin{proof}
  The product of~\(\alpha\otimes f\in\CT^{*}(X,U_{j};\kktilde)\) and \(\beta\otimes g\in\CT^{*}(\hatX_{i})\)
  is defined by
  \begin{equation}
    (\alpha\otimes f)\cup(\beta\otimes g) = \alpha\cup\pi^{*}(\hat\beta)\otimes fg,
  \end{equation}
  where \(\hat\beta\in C^{*}(X)\) is a preimage of~\(\beta\). 

  To show that this is well-defined,
  consider a \(\UU\)-small singular simplex~\(\sigma\) in~\(\tilde X\). If \(\sigma\) lies
  in~\(\pi^{-1}(U_{j})\), then so does any face~\(\sigma'\) of it. Hence \(\alpha(\sigma')=0\)
  and therefore \((\alpha\cup\pi^{*}(\hat\beta))(\sigma)=0\). 
  If \(\sigma\) does not lie in~\(\pi^{-1}(U_{j})\), then it lies in~\(\pi^{-1}(\hatX_{i})\supset\pi^{-1}(\hatX_{j})\) since it
  is \(\UU\)-small, and \((\alpha\cup\pi^{*}(\hat\beta))(\sigma)\) is again independent of
  the choice of~\(\hat\beta\).
  
  The commutativity of the diagram is clear by construction.
\end{proof}

For \(i\le j\), write
\begin{equation}
  \barCTc^{*}(X_{j},X_{i};\kktilde) = \dirlim \CT^{*}(U_{j},U_{i};\kktilde)
\end{equation}
where the direct limit is taken over all open covers~\(\UU\)
induced by sequences~\(U\) as above,
and let \(\barhCTc_{*}(X_{j},X_{i};\kktilde)\) be the \(R\)-dual complex.
Note that for~\(i\le j\le k\) we have short exact sequences
\begin{equation}
  \label{eq:XXj-XXi-XjXi}
  0 \to \barCTc^{*}(X_{k},X_{j};\kktilde) \to \barCTc^{*}(X_{k},X_{i};\kktilde) \to \barCTc^{*}(X_{j},X_{i};\kktilde) \to 0,
\end{equation}
and the canonical maps
\(\barCTc^{*}(X_{j},X_{i};\kktilde)\to\CTc^{*}(X_{j},X_{i};\kktilde)\)
and
\(\hCTc_{*}(X_{j},X_{i};\kktilde)\to\barhCTc_{*}(X_{j},X_{i};\kktilde)\)
are quasi-isomorphisms by tautness.

By passing to the direct limit in Lemma~\ref{thm:relative-cup-Ui}, we get
the family of relative cup products
\begin{equation}
  \label{eq:relative-cup-Xi}
  \barCTc^{*}(X,X_{i};\kktilde)\otimes\CT^{*}(\hatX_{i}) \stackrel{\cup}\longrightarrow \barCTc^{*}(X;\kktilde).
\end{equation}
Fix a representative~\(c_{T}\in\barhCTc_{n}(X;\kktilde)\)
of the equivariant orientation~\(o_{T}\in\hHTc_{n}(X;\kktilde)\).
Composition of~\eqref{eq:relative-cup-Xi} with~\(c_{T}\) yields
a pairing
\( 
  \barCTc^{*}(X,X_{i};\kktilde)\otimes\CT^{*}(\hatX_{i}) \to R
\), 
which we interpret as a map
\begin{equation}
  \label{eq:map-hCTXi-hCThatXi}
  f_{i}\colon\barCTc^{*}(X,X_{i};\kktilde)\to\hCT_{*}(\hatX_{i}).
\end{equation}

\begin{lemma}
  \label{thm:hCTXi-hCThatXi-iso}
  The map~\eqref{eq:map-hCTXi-hCThatXi} is a quasi-isomorphism.
  Moreover, for~\(i\le j\) it leads to a commutative diagram
  \begin{equation*}
    \begin{tikzcd} 
    0 \rar & \barCTc^{*}(X,X_{j};\kktilde) \dar{f_{j}} \rar & \barCTc^{*}(X,X_{i};\kktilde) \dar{f_{i}} \rar & \barCTc^{*}(X_{j},X_{i};\kktilde) \dar{f_{ji}} \rar & 0 \\
    0 \rar & \hCT_{*}(\hatX_{j}) \rar & \hCT_{*}(\hatX_{i}) \rar & \hCT_{*}(\hatX_{i},\hatX_{j}) \rar & 0 \mathrlap{,}
  \end{tikzcd}
  \end{equation*}  
  whose rows are exact and whose induced map~\(f_{ji}\) is a quasi-isomorphisms as well.
\end{lemma}

\begin{proof}
  The exactness of the top row in the diagram was already observed in~\eqref{eq:XXj-XXi-XjXi}.
  The compatibility of relative cup products with restrictions stated in Lemma~\ref{thm:relative-cup-Ui}
  implies that the left square in the diagram commutes, which induces the right vertical arrow.

  Note that Proposition~\ref{thm:relative-cohomology-complement}
  remains valid for twisted coefficients, and that the diagram
  \begin{equation}
  \begin{tikzcd} 
    \HTc^{*}(X,X_{i};\kktilde) \otimes \hHTc_{*}(X,X_{i};\kktilde) \arrow{r} \arrow{d}[left]{\cong} & R \arrow{d}{=} \\
    \HTc^{*}(\hatX_{i};\kktilde) \otimes \hHTc_{*}(\hatX_{i};\kktilde) \arrow{r} & R
  \end{tikzcd}
  \end{equation}
  is commutative.
  Moreover, the restriction of the orientation~\(o_{T}\) to any component of~\(\hatX_{i}\)
  is again an orientation.
  Hence the map~\eqref{eq:map-hCTXi-hCThatXi} corresponds
  to the map
  \( 
    \HTc^{*}(\hatX_{i};\kktilde) \to \hHT_{*}(\hatX_{i})
  \), 
  which is an isomorphism by Proposition~\ref{thm:PD-noncompact}.

  Coming back to the commutative ladder,
  two out of the three maps between the corresponding long exact sequences in (co)homology
  are isomorphisms, hence so is the third.
\end{proof}

\begin{proof}[Proof of Theorem~\ref{thm:PAL-A-B}]
  Consider the filtration \(\emptyset\subset B\subset A\subset X\) of~\(X\)
  and the associated diagram
  \begin{equation*}
    \begin{tikzcd} 
    0 \rar & \barCTc^{*}(A,B;\kktilde) \dar{f_{AB}} \rar & \barCTc^{*}(A;\kktilde) \dar{f_{A}} \rar & \barCTc^{*}(B;\kktilde) \dar{f_{B}} \rar & 0 \\
    0 \rar & \hCT_{*}(X\setminus B,X\setminus A) \rar & \hCT_{*}(X,X\setminus A) \rar & \hCT_{*}(X,X\setminus B) \rar & 0 \mathrlap{,}
  \end{tikzcd}
  \end{equation*}
  whose top row is again of the form~\eqref{eq:XXj-XXi-XjXi}.
  The maps~\(f_{A}\) and~\(f_{B}\) are special cases of the map~\(f_{ji}\) from Lemma~\ref{thm:hCTXi-hCThatXi-iso}.
  It follows from their definition that the right square commutes, which induces the map~\(f_{AB}\).
  By passing to (co)homology we get the commutative ladder stated in Theorem~\ref{thm:PAL-A-B}.
  Since \(H^{*}(f_{A})\) and \(H^{*}(f_{B})\) are isomorphisms by Lemma~\ref{thm:hCTXi-hCThatXi-iso},
  so is \(H^{*}(f_{AB})\). This proves the first part of the theorem.
  
  The analogous result with the roles of (co)homology reversed is obtained
  by applying the functor~\(\Hom_{R}(-,R)\) to the diagram above.
  Because the short sequences in the diagram split over~\(R\),
  their duals remain exact.
  Moreover, the natural inclusion of~\(\CT^{*}(A,B)\) into its double dual is a chain homotopy equivalence.
  This follows from the fact that for the chain-equivalent minimal Hirsch--Brown model,
  which is free and finitely generated over~\(R\), the corresponding map is even an isomorphism.
\end{proof}

A spectral sequence version of equivariant Poincaré--Alexander--Lefschetz duality
is as follows:

\begin{proposition}
  \label{thm:PAL-duality}
  Let \(o_{T}\in\hHTc_{n}(X;\kktilde)\) be an equivariant orientation of~\(X\).
  Taking the cap product with~\(o_{T}\)
  induces an isomorphism (of degree~\(-n\)) from the \(E_{1}\)~page on between the spectral sequences
  \begin{align*}
    E_{1}^{p} &= \HTc^{*}(X_{p},X_{p-1};\kktilde) \;\Rightarrow\;  \HTc^{*}(X;\kktilde), \\
    E_{1}^{p} &= \hHT_{*}(\hatX_{p-1},\hatX_{p}) \;\Rightarrow\;  \hHT_{*}(X).
  \end{align*}
  Similarly, the spectral sequences
  \begin{align*}
    E_{1}^{p} &= \hHTc_{*}(X_{p},X_{p-1};\kktilde) \;\Rightarrow\;  \hHTc_{*}(X;\kktilde), \\
    E_{1}^{p} &= \HT^{*}(\hatX_{p-1},\hatX_{p}) \;\Rightarrow\;  \HT^{*}(X)
  \end{align*}
  are isomorphic from the \(E_{1}\)~page on.
\end{proposition}

\begin{proof}
  We filter \(\barCTc^{*}(X;\kktilde)\) by~\(\FF_{i} = \barCTc^{*}(X,X_{i-1};\kktilde)\) for~\(0\le i\le m\)
  and similarly \(\hCT_{*}(X)\) by~\(\widehat\FF_{i} = \hCT_{*}(\hatX_{i-1})\).
  We know from Lemma~\ref{thm:hCTXi-hCThatXi-iso}
  that the diagram
  \begin{equation}
    \begin{tikzcd}
      \FF_{j} = \barCTc^{*}(X,X_{j};\kktilde) \arrow{r} \arrow{d} & \hCT_{*}(\hatX_{j}) = \widehat\FF_{j} \arrow{d} \\
      \FF_{i} = \barCTc^{*}(X,X_{i};\kktilde) \arrow{r} & \hCT_{*}(\hatX_{i}) = \widehat\FF_{i} 
    \end{tikzcd}
  \end{equation}
  commutes for~\(i\le j\),
  so that we obtain a map of spectral sequences with
  \begin{align}
    E_{0}^{i}(\FF) = \barCTc^{*}(X_{i},X_{i-1};\kktilde)
    &\to
    E_{0}^{i}(\widehat\FF) = \hCT_{*}(\hatX_{i-1},\hatX_{i}), \\
    \label{eq:PAL-E1}
    E_{1}^{i}(\FF) = \HTc^{*}(X_{i},X_{i-1};\kktilde)
    &\to
    E_{1}^{i}(\widehat\FF) = \hHT_{*}(\hatX_{i-1},\hatX_{i}).
  \end{align}
  It follows as in the proof of Theorem~\ref{thm:PAL-A-B}
  that the map~\eqref{eq:PAL-E1} is an isomorphism.
  The second part 
  follows analogously
  by dualizing \eqref{eq:map-hCTXi-hCThatXi} and the filtrations
  \(\FF\)~and~\(\widehat\FF\).
\end{proof}

Equipped with equivariant Poincaré--Alexander--Lefschetz duality,
we can easily deduce the following result, which is asserted
in~\cite[p.~849]{Bredon:1974} without proof.

\begin{corollary}
  \label{thm:quotient-homology-mf}
  If \(X\) is orientable and \(T\) acts locally freely,
  then \(X/T\) is an orientable \(\kk\)-homology manifold of dimension~\(n-r\).
\end{corollary}

\begin{proof}
  As discussed in Remark~\ref{rem:quotient}, 
  \(X/T\) satisfies our assumption on spaces, and it is connected since \(X\) is.

  To verify condition~\eqref{eq:def-homology-mf},
  take an~\(x\in X\) with image~\(\bar x\in \bar X=X/T\).
  By Proposition~\ref{thm:locally-free-action}
  and Poincaré--Alexander--Lefschetz duality
  for the \(T\)-pair~\((Tx,\emptyset)\) in~\(X\), we have
  \begin{align}
    H_{i}(\bar X,\bar X\setminus\{\bar x\})
    &=\hHT_{i+r}(X,X\setminus Tx)
    \cong \HTc^{n-r-i}(Tx) \\
    \notag
    &= \Hc^{n-r-i}(\{\bar x\})
    = \begin{cases}
      \kk & \text{if \(i=n-r\),} \\
      0 & \text{otherwise.}
    \end{cases}
  \end{align}

  Again by Proposition~\ref{thm:locally-free-action},
  the equivariant orientation~\(o_{T}\in\hHTc_{n}(X)\) descends to
  a non-zero element in~\(\hHc_{n-r}(X/T)\).
  Hence \(X/T\) is orientable.
\end{proof}

\begin{example}
  A simple example shows why orientability is needed in Corollary~\ref{thm:quotient-homology-mf}
  above. Let \(X\) be the open Möbius band with its standard locally free
  action of~\(T=S^{1}\). Then \(X/T\) is a half-open interval, and so it is
  not a (homology) manifold, but rather a manifold with boundary. The quotient~\(\tilde X/T\)
  of the orientation cover looks like the letter~``V'' with its vertex
  corresponding to the end point of the interval, which in turn
  corresponds to the middle circle, the only non-free orbit.
\end{example}

\begin{remark}
  \label{rem:locally-free}
  Let \((A,B)\) be a closed \(T\)-pair in~\(X\).
  In Proposition~\ref{thm:locally-free-action} we established
  an isomorphism of \(H^{*}(BL)\)-modules
  \begin{equation}
    \label{eq:locally-free-action-pal}
    \hHT_{*}(A,B) = H^{L}_{*-p}(A/K,B/K)
  \end{equation}
  whenever a subtorus~\(K\subset T\) of rank~\(p\) and with quotient~\(L=T/K\)
  acts freely on~\(A\setminus B\);
  a locally free action was sufficient in case~\(\Char\kk=0\). 
  In the context of orientable homology manifolds, we can now understand this isomorphism
  in terms of Poincaré--Alexander--Lefschetz duality:

  Assume that \(K\) acts freely (or just locally freely if \(\Char\kk=0\)) on the orientable homology manifold~\(X\),
  so that \(X/K\) is again an orientable homology manifold by Corollary~\ref{thm:quotient-homology-mf}.
  Let \(n=\dim X=\dim X/K+p\).
  Using the cohomological part of Proposition~\ref{thm:locally-free-action}
  and Poincaré--Alexander--Lefschetz duality, we get
  \begin{multline}
    \hHT_{*}(A,B) = \HTc^{n-*}(X\setminus B,X\setminus A) \\
    = H_{L,c}^{n-*}((X\setminus B)/K,(X\setminus A)/K)
    = H^{L}_{*-p}(A/K,B/K).
  \end{multline}
  Hence the isomorphism~\eqref{eq:locally-free-action-pal} can be interpreted
  as a push-forward map or integration over the fibre in this setting.
\end{remark}

\subsection{Thom isomorphism}

As in the non-equivariant case, the Thom isomorphism
is a consequence of Poincaré and Poincaré--Alexander--Lefschetz duality,
\cf~\cite[\S VIII.7, \S VIII.11]{Dold:1980}.
In fact, one can use our version of equivariant duality
to define also Gysin homomorphisms (push forwards), indices, Euler classes
etc.\ in the equivariant setting and to prove their main properties
(\cf~\cite[Sec.~5.3]{AlldayPuppe:1993}) for cohomology with different supports.
The use of the Cartan model even provides a more functorial approach
than the minimal Hirsch--Brown model used in~\cite{AlldayPuppe:1993}.
Here we only develop the theory as far as needed for our applications in Section~\ref{applications-homology-mf}.

We continue to assume that \(X\) is an \(n\)-dimensional 
\(\kk \)-homology manifold with a \(T\)-action.

\begin{proposition}
  \label{thm:thom-iso}
  $ $
  \begin{enumerate}
  \item Let \(Y\subset X\)
    be a closed \(T\)-stable \(\kk\)-homology manifold of dimension~\(m\).
    Suppose that the orientation cover of~\(X\) restricts to the orientation cover of~\(Y\).
    Then there is an isomorphism of \(R\)-modules
    \begin{equation*}
      \HT^{*}(X,X\setminus Y) \cong \HT^{*}(Y)
    \end{equation*}
    of degree~\(m-n\).
  \item Assume \(\Char\kk=0\), and let \(K\subset T\) be a subtorus. Then there is an isomorphism of \(R\)-modules
    \begin{equation*}
      \HT^{*}(X,X\setminus X^{K}) \cong \HT^{*}(X^{K}).
    \end{equation*}
    This isomorphism has degree~\(m-n\) if all components of~\(X^{K}\)
    are of dimension~\(m\); in general it only preserves degrees mod~\(2\).
  \end{enumerate}
\end{proposition}

\begin{proof}
  We start with the first case.
  By Poincaré--Alexander--Lefschetz duality for the pair~\((X,Y)\) and
  Poincaré duality for~\(Y\) we have isomorphisms of \(R\)-modules
  \begin{equation}
    \HT^{*}(X,X\setminus Y) \cong \hHTc_{*}(Y,\kktilde) \cong \HT^{*}(Y),
  \end{equation}
  whose composition has degree~\(m-n\).
  Note that for the first isomorphism \(\hHTc_{*}(Y,\kktilde)\)
  is defined via the restriction of the orientation cover of~\(X\),
  and via the orientation cover for~\(Y\) in the second isomorphism.
  By assumption, these two covers coincide.

  We now consider the fixed point set~\(X^{K}\).
  It has finitely many components, say \(Y_{1}\),~\ldots,~\(Y_{k}\),
  which are \(\kk\)-homology manifolds whose dimensions are congruent to~\(n\) mod~\(2\)
  by a result of Conner and Floyd~\cite[Thm.~V.3.2]{Borel:1960}.
  By excision we have
  \begin{equation}
    \HT^{*}(X,X\setminus X^{K}) = \bigoplus_{i} \HT^{*}(X,X\setminus Y_{i}).
  \end{equation}
  The claim follows once we know
  that the restriction of an orientation cover for~\(X\)
  to each~\(Y_{i}\) is an orientation cover for that component.
  This is the content of the following lemma. 
\end{proof}

\begin{lemma}
  \label{thm:orientation-cover-XT}
  Assume \(\Char\kk=0\).
  Then the restriction of an orientation cover for~\(X\)
  to any component~\(Y\) of~\(X^{T}\) is an orientation cover for~\(Y\).
\end{lemma}

Note that each component~\(Y\) is orientable if and only if its orientation cover is trivial.
According to the theorem of Conner and Floyd mentioned previously, each component~\(Y\) of~\(X^{T}\) is orientable if so is \(X\).
Lemma~\ref{thm:orientation-cover-XT} can therefore be seen as a generalization of this part of their result.
Also note that for a smooth \(T\)-manifold~\(X\) Lemma~\ref{thm:orientation-cover-XT}
is a consequence of the fact that the normal bundle of each component~\(Y\)
of~\(X^{T}\) is orientable, \cf~\cite[Cor.~2]{Duflot:1983}:
By excision one can restrict from~\(X\) to a \(T\)-stable tubular neighbourhood of~\(Y\),
and, like the normal bundle, this neighbourhood is orientable if and only if \(Y\) is.

\begin{proof}
  Let \(\tilde X\to X\) be an orientation cover for~\(X\) and \(\tilde Z\to Z\)
  its restriction to~\(Z=X^{T}\). Note that \(\tilde Z=(\tilde X)^{T}\).
  For each component~\(Y\) of~\(Z\), say of dimension~\(m\), let \(\tilde Y\to Y\) be the further restriction.
  We will show \(\hHc_{m}(\tilde Y)_{-}\ne0\), which proves that \(\tilde Y\to Y\)
  is an orientation cover of~\(Y\): If \(Y\) is orientable, this condition ensures that \(\tilde Y\)
  is disconnected, and if \(Y\) is non-orientable, it shows that \(\tilde Y\) is orientable.
  
  Since the cap product~\eqref{eq:cap-product-twisted-c} is natural with respect to proper maps of spaces,
  we get a commutative diagram
  \begin{equation}
    \begin{tikzcd}
      \HTc^{*}(\tilde X)_{-} \arrow{r}{\cap\,\iota_{*}(b)} \arrow{d}[left]{\iota^{*}} & \hHT_{*}(\tilde X)_{+} \\
      \HTc^{*}(\tilde Z)_{-} \arrow{r}{\cap\,b} & \hHT_{*}(\tilde Z)_{+}\mathrlap{,}  \arrow{u}[right]{\iota_{*}}
    \end{tikzcd}
  \end{equation}
  where \(\iota\colon\tilde Z\hookrightarrow\tilde X\) and \(b\in\hHTc_{*}(\tilde Z)_{-}\).
  Note that \(\iota_{*}\colon \hHc_{*}(\tilde Z)\to\hHc_{*}(\tilde X)\)
  commutes with the involution~\(\tau\) and therefore preserves the \(\pm1\)~eigenspaces.
  Let \(S\subset R\) be the multiplicative subset of homogeneous polynomials of positive degree.
  We localize the diagram at~\(S\) and choose \(b\) to be a preimage of~\(o_{T}\in S^{-1}\hHT_{*}(\tilde X)_{-}\),
  which is possible by the localization theorem in equivariant homology (Proposition~\ref{thm:localization-thm-homology},
  here for homology with closed supports).
  By the same result and equivariant Poincaré duality, this
  turns the top and vertical arrows into isomorphisms, hence also the bottom arrow.

  Now \(b\in S^{-1}\hHTc_{*}(\tilde Z)_{-}\) is a sum of elements, one for each component of~\(Z=X^{T}\).
  The summand~\(b^{Y}\) corresponding to the component~\(Y\)
  can be written in the form
  \begin{equation}
    b^{Y} = b^{Y}_{m} + \dots + b^{Y}_{0}
    \in S^{-1}\hHTc_{*}(\tilde Y)_{-} = \hHc_{*}(\tilde Y)_{-}\otimes S^{-1}R
  \end{equation}
  for some~\(b^{Y}_{i}\in \hHc_{i}(\tilde Y)_{-}\otimes S^{-1}R\).
  A cap product \(\alpha\cap c\) with \(\alpha\in H^{m}(\tilde Y)\)
  and \(c\in \hHc_{i}(\tilde Y)\) vanishes unless \(i = m\).
  Because capping with~\(b^{Y}\) is an isomorphism,
  we conclude that \(b^{Y}_{m}\ne0\), hence~\(\hHc_{m}(Y)_{-}\ne0\).
\end{proof}

\section{Applications to the orbit structure}
\label{sec:PAL-applications}

We assume throughout the rest of this paper that
\(X\) is a \(T\)-space and that
the characteristic of the field~\(\kk\) is \(0\).
Recall that the orbit filtration~\((X_{i})\) has been defined in the introduction.

\subsection{\texorpdfstring{\boldmath{General \(T\)-spaces}}{General T-spaces}}

In Sections \ref*{sec:main-result},~\ref*{sec:applications}.1 and~\ref*{sec:partial-exactness}
of~\cite{AlldayFranzPuppe} we established results about the
equivariant cohomology with closed supports and
equivariant homology with compact supports
of the orbit filtration of a \(T\)-space~\(X\).
All these results have analogues for the other pair of supports,
\ie, for cohomology with compact supports and homology with closed supports.
Moreover, for a \(\kk\)-homology manifold~\(X\), one has another set of
analogous results for (co)homology with twisted coefficients.
The proofs for the new cases are usually identical to the ones given in~\cite{AlldayFranzPuppe}.
In the case of twisted coefficients, one may alternatively derive them
from the decompositions~\eqref{eq:splitting-HT} and~\eqref{eq:splitting-hHT}
and the untwisted result for an orientation cover;
see Proposition~\ref{thm:stratum-cm-c} below for an example.
We therefore content ourselves by stating the most important results in a more general setting.
All results in this section
are equally valid for the other pair of supports.

We simplify notation in the following way: 
For a \(T\)-pair~\((A,B)\) in a homology manifold~\(X\) we write \(\HT^{*}(A,B;\AA)\) to denote
either cohomology with constant coefficients (\(\AA=\kk\)) or
with twisted coefficients (\(\AA=\kktilde\)).
The same applies to homology and (co)chain complexes.
If \(X\) is not a homology manifold, then \(\AA\) always means constant coefficients.

\begin{proposition}
  \label{thm:stratum-cm-c}
  The \(R\)-modules \(\HT^{*}(X_{i},X_{i-1};\AA)\)~and~\(\hHT_{*}(X_{i},X_{i-1};\AA)\) are zero or
  Cohen--Macaulay of dimension~\(r-i\) for~\(0\le i\le r\). 
\end{proposition}

\begin{proof}
  The version for constant coefficients and the usual pair of supports
  is proved in~\citeorbitsone{Prop.~\ref*{thm:stratum-cm}},
  following the ideas of~\cite[Sec.~7]{Atiyah:1974}.
  The proof for the other pair of supports is identical.
  The case of twisted coefficients follows
  from the untwisted version for an orientation cover
  and the observation that a non-zero direct summand
  of a Cohen--Macaulay module is again Cohen--Macaulay of the same dimension.
\end{proof}

\begin{corollary}
  \label{thm:hHT-orbit-degeneration-c}
  The spectral sequence associated with the orbit filtration of \(\hCT_{*}(X;\AA)\)
  and converging to~\(\hHT_{*}(X;\AA)\)
  degenerates at~\(E^{1}_{p}=\hHT_{*}(X_{p},X_{p-1};\AA)\).
\end{corollary}

\begin{proof}
  See \citeorbitsone{Cor.~\ref*{thm:hHT-orbit-degeneration}}.
\end{proof}

The following two results are immediate consequences of Corollary~\ref{thm:hHT-orbit-degeneration-c},
\cf~\citeorbitsone{Cor.~\ref*{thm:Ext-hHT-zero}}.
For the convenience of the reader, we provide proofs that are based only on
the crucial Cohen--Macaulay property identified in Proposition~\ref{thm:stratum-cm-c}.

\begin{proposition}
  \label{thm:hHT-short-exact}
  For any~\(-1\le i< j\le r\) there is a short exact sequence
  \begin{equation*}
    0 \longrightarrow \hHT_{*}(X_{j},X_{i};\AA)
    \longrightarrow \hHT_{*}(X,X_{i};\AA)
    \longrightarrow \hHT_{*}(X,X_{j};\AA)
    \longrightarrow 0.
  \end{equation*}
\end{proposition}

\begin{proposition}
  \label{thm:Ext-i-j-0}
  \(\Ext_{R}^{p}(\hHT_{*}(X_{j},X_{i};\AA),R) = 0\) for~\(p>j\) and~\(p\le i\).
  In other words, 
  \(\dim_{R}\hHT_{*}(X_{j},X_{i};\AA)\le r-i-1\)
  and \(\depth_{R}\hHT_{*}(X_{j},X_{i};\AA)\le r-j\).
\end{proposition}

\begin{proof}[Proof of Propositions~\ref{thm:hHT-short-exact} and~\ref{thm:Ext-i-j-0}]
  We prove both statements simultaneously by falling induction on~\(i\).
  For \(i=r\) there is nothing to show.

  Now assume both claims are true for a given~\(i\) and all~\(j\ge i\).
  By Proposition~\ref{thm:stratum-cm-c}, \(\hHT_{*}(X_{i},X_{i-1};\AA)\) is zero or Cohen--Macaulay of dimension~\(r-i\).
  Because \(\hHT_{*}(X,X_{i};\AA)\) is of dimension~\(\le r-i-1\) by induction,
  the connecting homomorphism
  \begin{equation}
    \hHT_{*}(X,X_{i};\AA) \to \hHT_{*-1}(X_{i},X_{i-1};\AA)
  \end{equation}
  is zero, \cf~\citeorbitsone{Lemma~\ref*{thm:CM-map-0}}, so that we get the short
  exact sequence
  \begin{equation}
    \label{eq:hHT-short-exact-1}
    0 \longrightarrow \hHT_{*}(X_{i},X_{i-1};\AA)
    \longrightarrow \hHT_{*}(X,X_{i-1};\AA)
    \longrightarrow \hHT_{*}(X,X_{i};\AA)
    \longrightarrow 0.
  \end{equation}
  By induction, the map~\(\hHT_{*}(X,X_{i};\AA)\to\hHT_{*}(X,X_{j};\AA)\)
  is surjective, hence so is the composition
  \begin{equation}
    \hHT_{*}(X,X_{i-1};\AA) \to \hHT_{*}(X,X_{i};\AA)\to\hHT_{*}(X,X_{j};\AA),
  \end{equation}
  which proves the first claim.
  Taking \(X=X_{j}\) in~\eqref{eq:hHT-short-exact-1}, we obtain
  \begin{equation}
    0 \longrightarrow \hHT_{*}(X_{i},X_{i-1};\AA)
    \longrightarrow \hHT_{*}(X_{j},X_{i-1};\AA)
    \longrightarrow \hHT_{*}(X_{j},X_{i};\AA)
    \longrightarrow 0.
  \end{equation}
  The second claim now follows by induction and the way \(\Ext\)~modules
  (or dimension and depth{\slash}projective dimension) behave with respect to
  short exact sequences.
\end{proof}

The spectral sequence for equivariant cohomology induced by the orbit filtration
does \emph{not} degenerate at the \(E_{1}\)~page in general.
Since this page of the spectral sequence
is of independent interest, we give it a name.

The \emph{non-augmented Atiyah--Bredon complex}~\(\AB^{*}(X;\AA)\)
with coefficients in~\(\AA\) is the complex of \(R\)-modules
defined by
\begin{equation}
 \AB^{i}(X;\AA)=\HT^{*+i}(X_{i}, X_{i-1};\AA) 
\end{equation}
for~\(0\le i\le r\)
and zero otherwise.
The differential
\begin{equation}
  d_{i}\colon \HT^{*}(X_{i}, X_{i-1};\AA)\to \HT^{*+1}(X_{i+1}, X_{i};\AA)
\end{equation}
is the connecting morphism in the long exact sequence of the triple~\((X_{i+1},X_{i},X_{i-1})\).
Note that \(\AB^{*}(X;\AA)\) is the \(E_{1}\)~page of the spectral sequence
arising from the orbit filtration of~\(X\) and converging to~\(\HT^{*}(X;\AA)\),
and its cohomology~\(H^{*}(\AB^{*}(X;\AA))\) is the \(E_{2}\)~page.

The \emph{augmented Atiyah--Bredon complex} 
is obtained by augmenting~\(\AB^{*}(X;\AA)\)
by~\(\AB^{-1}(X;\AA)=\HT^{*}(X;\AA)\) and the restriction to the fixed point set,
\begin{multline}
  \label{eq:4:atiyah-bredon}
  0
  \longrightarrow \HT^{*}(X;\AA)
  \longrightarrow \HT^{*}(X_0;\AA)
  \stackrel{d_{0}}\longrightarrow \HT^{*+1}(X_1, X_0;\AA)
  \stackrel{d_{1}}\longrightarrow \cdots \\ \cdots
  \stackrel{d_{r-1}}\longrightarrow \HT^{*+r}(X_r, X_{r-1};\AA)
  \longrightarrow 0.
\end{multline}

\begin{remark}
This sequence first appeared explicitly in the paper~\cite{Bredon:1974} of Bredon,
but it goes back to work of Atiyah~\cite[Sec.~7]{Atiyah:1974}.
In the context of equivariant \(K\)-theory,
Atiyah showed that the freeness of~\(K_{T}^{*}(X)\) implies that the sequence
\begin{equation}
 0 \to K_{T}^{*}(X,X_{i-1}) \to K_{T}^{*}(X_{i},X_{i-1}) \to K_{T}^{*}(X,X_{i}) \to 0
\end{equation}
is exact for all~\(i\) \cite[eq.~(7.3)]{Atiyah:1974}.
This in turn is equivalent to the exactness of the \(K\)-theoretic analogue of~\eqref{eq:4:atiyah-bredon},
\cf~\cite[Lemma~4.1]{FranzPuppe:2007}.
Atiyah actually considered representations only, but his arguments work for any \(T\)-space.
\end{remark}

It turns out that the cohomology of the non-augmented Atiyah--Bredon complex
is completely determined by~\(\hHT_{*}(X;\AA)\).

\goodbreak

\begin{theorem}
  \label{thm:exthab-ss-c}
  For any \(T\)-space~\(X\) 
  the following two spectral sequences converging to~\(\HT^{*}(X;\AA)\) are naturally isomorphic from the \(E_{2}\)~page on:
  \begin{enumerate}
  \item The one induced by the orbit filtration with~\(E_{1}^{p}=\HT^{*}(X_{p},X_{p-1};\AA)\),
  \item The universal coefficient spectral sequence with~\(E_{2}^{p}=\Ext_{R}^{p}(\hHT_{*}(X;\AAv),R)\).
  \end{enumerate}
\end{theorem}

\begin{proof}
  See \citeorbitsone{Thm.~\ref*{thm:exthab-ss}}.
  The version for twisted coefficients may again be derived from the untwisted
  result for an orientation cover.
\end{proof}

\begin{corollary}
  \label{thm:4:ext-hab}
    For any~\(i\ge0\) there is an isomorphism of \(R\)-modules
  \begin{equation*}
    H^{i}(\AB^{*}(X;\AA)) = \Ext_{R}^{i}(\hHT_{*}(X;\AA),R).
  \end{equation*}
\end{corollary}

In Section~\ref{sec:quick-proof} we will give a direct proof of this important result
that is not based on Theorem~\ref{thm:exthab-ss-c}.

\begin{theorem}
  \label{thm:4:conditions-partial-exactness}
  The following conditions are equivalent for any~\(0\le j\le r\):
  \begin{enumerate}
  \item \label{4:q1} The Atiyah--Bredon sequence~\eqref{eq:4:atiyah-bredon}
    is exact at all positions~\(-1\le i\le j-2\).
  \item \label{4:q4} The restriction map~\(\HT^{*}(X;\AA)\to H_{K}^{*}(X;\AA)\)
    is surjective for all subtori~\(K\) of~\(T\) of rank~\(r-j\).
  \item \label{4:q3} \(\HT^{*}(X;\AA)\) is free over all subrings~%
    \(H^{*}(BL)\subset H^{*}(BT)=R\), where \(L\) is a quotient of~\(T\) of rank~\(j\).
  \item \label{4:q2} \(\HT^{*}(X;\AA)\) is a \(j\)-th syzygy.
  \end{enumerate}
\end{theorem}

Several equivalent definitions of syzygies are collected in~\citeorbitsone{Sec.~\ref*{sec:Torsion-freeness}}.

\begin{proof}
  The proof of~\citeorbitsone{Thm.~\ref*{thm:conditions-partial-exactness}} carries over.
  Only the argument for the equivalence~\(\hbox{\eqref{4:q4}}\Leftrightarrow\hbox{\eqref{4:q3}}\)
  has to be slightly modified in the case of twisted coefficients:
  The involution on an orientation cover~\(\tilde X\) induces one on the Borel construction~\(\tilde X_{T}\),
  and \(\HT^{*}(X;\kktilde)=H^{*}(\tilde X_{T})_{-}\) in the notation of Section~\ref{sec:twisted},
  and analogously for~\(K\).
  (Note that the decomposition of the cohomology into the \(\pm1\)~eigenspaces of the involution
  exists even for spaces that do not satisfy our standing assumptions.)
  Now one considers the map
  \begin{equation}
    \HT^{*}(X;\kktilde)=H_{T/K}^{*}(\tilde X_{K})_{-} \to H^{*}(\tilde X_{K})_{-}=H_{K}^{*}(X;\kktilde).
  \end{equation}
  and applies the Leray--Hirsch argument
  as used in~\cite{AlldayFranzPuppe}
  to the \(-1\)~eigenspaces.
\end{proof}

\subsection{Homology manifolds}
\label{applications-homology-mf}

In this section we assume that \(X\) is a \(\kk\)-homology manifold.
Theorem~\ref{thm:exthab-ss-c}
and \citeorbitsone{Thm.~\ref*{thm:exthab-ss}}
may be combined
with Poincaré duality and Poincaré--Alexander--Lefschetz duality
in various ways. The following result is an example of this.
Recall that \(\hatX_{i}=X\setminus X_{i}\).

\begin{corollary}
  The following spectral sequences
  are isomorphic from the \(E_{2}~\)page on:
  \begin{align*}
    E_{1}^{p} &= \hHT_{*}(\hatX_{p-1},\hatX_{p}) \;\Rightarrow\; \hHT_{*}(X), \\
    E_{2}^{p} &= \Ext_{R}^{p}(\HT^{*}(X),R)  \;\Rightarrow\; \hHT_{*}(X).
  \end{align*}
\end{corollary}

\begin{proof}
  Let \(n=\dim X\). By Theorem~\ref{thm:PAL-duality},
  the first spectral sequence is isomorphic,
  from the \(E_{1}\)~page on, to the spectral sequence
  \begin{equation}
    E_{1}^{p}=\HTc^{*}(X_{p},X_{p-1};\kktilde)[-n] \;\Rightarrow\; \HTc^{*}(X;\kktilde)[-n].
  \end{equation}
  By Poincaré duality, 
  the second spectral sequence is isomorphic to
  \begin{equation}
    E_{2}^{p}=\Ext_{R}^{p}(\hHTc_{*}(X;\kktilde),R)[-n] \;\Rightarrow\; \HTc^{*}(X;\kktilde)[-n].
  \end{equation}
  Hence the claim follows from Theorem~\ref{thm:exthab-ss-c}.
\end{proof}

\begin{proposition}
  \label{thm:thom-iso-Xi-Xi1}
  For any~\(0\le i\le r\) there is an isomorphism of \(R\)-modules
  \begin{equation*}
    \HT^{*}(\hatX_{i-1},\hatX_{i})\cong\HT^{*}(X_{i}\setminus X_{i-1}),
  \end{equation*}
  preserving degrees modulo~\(2\).
\end{proposition}

\begin{proof}
  Since only finitely many
  isotropy groups occur in~\(X\),
  there is a subtorus~\(K\subset T\)
  such that \(\hatX_{i-1}^{K}=X_{i}\setminus X_{i-1}\).
  Hence our claim reduces to the Thom isomorphism from Proposition~\ref{thm:thom-iso}.
\end{proof}

The following result generalizes a theorem of Duflot~\cite[Thm.~1]{Duflot:1983}
concerning smooth actions on differential manifolds.
More than the extension to continuous actions on homology manifolds, our main insight is
that Duflot's result follows
by equivariant Poincaré--Alexander--Lefschetz duality
from
Proposition~\ref{thm:hHT-short-exact}
which is valid for \emph{all} \(T\)-spaces.

\begin{proposition}
  \label{thm:duflot-general}
  For any~\(0\le i\le r\)
  there are short exact sequences
  \begin{equation*}
    0 \to \HT^{*}(X,\hatX_{i}) \to \HT^{*}(X) \to \HT^{*}(\hatX_{i}) \to 0
  \end{equation*}
  and
  \begin{equation*}
    0 \to \HT^{*}(X,\hatX_{i-1}) \to \HT^{*}(X,\hatX_{i}) \to \HT^{*}(X_{i}\setminus X_{i-1}) \to 0
  \end{equation*}
  where the right map in the lower sequence preserves degrees only mod~\(2\).
\end{proposition}

Duflot also considers actions of \(p\)-tori~\((\Z_{p})^{r}\) with~\(p>2\).
The results of~\cite{AlldayFranzPuppe} and this paper can as well be extended to \(p\)-tori;
we will elaborate on this elsewhere because some proofs require modification.

\begin{proof}
  By Proposition~\ref{thm:hHT-short-exact} we have a short exact sequence
  \begin{equation}
    \label{eq:Xi-X-short-exact}
    0 \longrightarrow \hHT_{*}(X_{i};\kktilde)
    \longrightarrow \hHT_{*}(X;\kktilde)
    \longrightarrow \hHT_{*}(X,X_{i};\kktilde)
    \longrightarrow 0.
  \end{equation}
  The first short exact sequence we are claiming follows from this
  by Poincaré--Alexander--Lefschetz duality (Theorem~\ref{thm:PAL-A-B}).

  Replacing \(X_{i}\) by~\(X_{i-1}\) and \(X_{j}\) by~\(X_{i}\)
  in Proposition~\ref{thm:hHT-short-exact}
  leads similarly to the short exact sequence
  \begin{equation}
    0 \to \HT^{*}(X,\hatX_{i-1}) \to \HT^{*}(X,\hatX_{i}) \to \HT^{*}(\hatX_{i-1},\hatX_{i}) \to 0.
  \end{equation}
  Combining this with Proposition~\ref{thm:thom-iso-Xi-Xi1}
  confirms our second claim.
\end{proof}

Not surprisingly, we also get the following spectral sequence version:

\begin{proposition}
  The spectral sequence associated
  to the filtration~\((\hatX_{i})\) and converging to~\(\HT^{*}(X)\)
  degenerates at the \(E_{1}\)~page.
\end{proposition}

\begin{proof}
  By Theorem~\ref{thm:PAL-duality}, this spectral sequence is isomorphic,
  from the \(E_{1}\)~page on, to the spectral sequence converging to~\(\hHTc_{*}(X;\kktilde)\)
  with \(E_{1}^{p}=\hHTc_{*}(X_{p},X_{p-1};\kktilde)\). The latter degenerates
  by Corollary~\ref{thm:hHT-orbit-degeneration-c}.
\end{proof}

\subsection{Uniform actions}

Let \(X\) be a \(T\)-space.
For dimensional reasons, it follows from Proposition~\ref{thm:stratum-cm-c}
that the differential
\begin{equation}
   d_{i}\colon \HT^{*}(X_{i}, X_{i-1})\to \HT^{*+1}(X_{i+1}, X_{i})
\end{equation}
cannot be injective unless \(\HT^{*}(X_{i}, X_{i-1})=0\).
This has implications for the uniformity of actions,
which we discuss now.

Recall from~\cite[Def.~3.6.17]{AlldayPuppe:1993}
that the \(T\)-action on~\(X\) is said to be \emph{uniform}
if for any subtorus~\(K\subset T\) and any component~\(F\) of~\(X^{K}\)
one has \(F^{T}\ne\emptyset\).
(This implies \(X^{T}\ne\emptyset\) if \(X\ne\emptyset\).)

We call \(F\) a \emph{minimal stratum} of~\(X\) corresponding to the subtorus~\(K\subset T\)
if \(F\) is a component of~\(X^{K}\) 
and if \(F^{L}=\emptyset\) for any subtorus~\(L\) properly containing \(K\).
Note that the action is uniform if and only if all minimal strata are components of~\(X^{T}\).
This observation makes it easy to construct non-uniform actions, even in the context
of compact orientable manifolds with fixed points, see \cite[Ex.~1.7.4]{Allday:2005}.

It has been noted by a number of authors
that the \(T\)-action is uniform if \(\HT^{*}(X)\) is a free \(R\)-module.
A large part of~\cite{AlldayFranzPuppe}, however, is concerned
with the case where \(\HT^{*}(X)\) is a torsion-free \(R\)-module (that is, a first syzygy),
but not necessarily free. So we note the following, which is also an immediate
consequence of the characterization of uniform actions given in~\cite[Thm.~3.6.18]{AlldayPuppe:1993}.

\begin{proposition}
  If \(\HT^{*}(X)\) is \(R\)-torsion-free, then the action is uniform.
\end{proposition}

\begin{proof}
  Assume that there is a minimal stratum~\(F\), corresponding to a subtorus \(K\subsetneq T\).
  Then \(\HT^{*}(F)\) is a direct summand of~\(\HT^{*}(X^{K})\). Set \(S=H^{*}(BK)\setminus\{0\}\)
  and \(\tilde S=R\setminus\{0\}\).
  By the localization theorem, we have
  \begin{equation}
    S^{-1}\HT^{*}(X) \cong S^{-1}\HT^{*}(X^{K}) = S^{-1}\HT^{*}(F)\oplus S^{-1}\HT^{*}(X^{K}\setminus F),
  \end{equation}
  so we can choose a~\(c\in\HT^{*}(X)\) such that its image in~\(S^{-1}\HT^{*}(X^{K})\)
  is non-zero and lies in~\(S^{-1}\HT^{*}(F)\). Because \(F^{T}\) is empty, \(\HT^{*}(F)\)
  is \(R\)-torsion. This implies that the image of~\(c\) in~\(\tilde S^{-1}\HT^{*}(X^{T})\) is zero,
  hence also the one in~\(\HT^{*}(X^{T})\). But this is a contradiction
  because the torsion-freeness of~\(\HT^{*}(X)\) is equivalent
  to the injectivity of the map~\(\HT^{*}(X)\to\HT^{*}(X^{T})\).
\end{proof}

\begin{proposition}
  \label{thm:minimal-stratum-HAB}
  Let \(F\) be a minimal stratum of~\(X\) corresponding to a subtorus \(K\subset T\) of rank~\(r-i\).
  Then \(H^{i}(\AB^{*}(X))\ne0\).
\end{proposition}

\begin{proof}
  The minimal stratum~\(F\) is a component of both~\(X_{i}\setminus X_{i-1}\)
  and \(X_{i}\setminus X_{i-2}\). So the summand \(\HT^{*}(F)\)
  maps isomorphically under the restriction~\(\HT^{*}(X_{i},X_{i-1})\to\HT^{*}(X_{i},X_{i-2})\),
  and \(\HT^{*}(F)\cap\im d_{i-1}=0\).
  On the other hand, \(\dim_{R}\HT^{*}(F)=r-i\).
  Because \(\HT^{*}(X_{i+1}, X_{i})\) is of dimension~\(r-i-1\), the restriction of the differential
  to~\(\HT^{*}(F)\) cannot be injective.
\end{proof}

Proposition~\ref{thm:minimal-stratum-HAB} also follows
from~\cite[Thm.~3.6.14]{AlldayPuppe:1993}.
For another result relating the uniformity and torsion-freeness,
see~\cite[Thm.~3.8.7\,(4)]{AlldayPuppe:1993}.
In the notation of that theorem, a minimal stratum~\(F=c\)
corresponding to a subtorus~\(K\subset T\) gives one of the pairs~\((K_{i},c_{i})\),
\(1\le i\le\gamma\).

\section{
The cohomology of the Atiyah--Bredon complex}
\label{sec:quick-proof}

In this section we shall give a direct proof of Corollary~\ref{thm:4:ext-hab}.
Instead of reasoning with spectral sequences, we will rely
on Propositions~\ref{thm:hHT-short-exact} and~\ref{thm:Ext-i-j-0}.
Our proof is valid for any pair of supports and, in case of a \(\kk\)-homology manifold,
also for twisted coefficients. For ease of notation, we write it down only
for constant coefficients and the usual pair of supports.
Recall that we are still assuming the characteristic of~\(\kk\) to be \(0\).
For convenience, we define \(X_{r+1}=X\) in addition to~\(X_{-1}=\emptyset\).

Let~\(0\le i\le r\).
The following commutative diagram with exact rows will play a central role:
\begin{equation}
\begin{tikzcd}[column sep=small]
  \label{eq:hHT-Xi-Xi1-diagram}
  0 \rar & \hHT_{*}(X_{i}) \dar \rar & \hHT_{*}(X_{i+1}) \dar \rar & \hHT_{*}(X_{i+1},X_{i}) \dar \rar & 0 \\
  0 \rar & \hHT_{*}(X_{i},X_{i-1}) \rar & \hHT_{*}(X_{i+1},X_{i-1}) \rar & \hHT_{*}(X_{i+1},X_{i}) \rar & 0.
\end{tikzcd}
\end{equation}
The exactness of top row is Proposition~\ref{thm:hHT-short-exact}
for the triple~\((X_{i+1},X_{i},X_{-1})\),
and that of the bottom follows by looking at~\((X_{i+1},X_{i},X_{i-1})\).

For brevity, we denote \(\hHT_{*}(X_{j},X_{i})\) by~\(M_{j,i}\) and,
for any \(R\)-module~\(M\), we abbreviate \(\Ext_{R}^{p}(M,R)[p]\) by~\(\EE^{p}(M)\).
From the bottom row of~\eqref{eq:hHT-Xi-Xi1-diagram} and the long exact sequence for~\(\Ext\)
we have 
a connecting homomorphism
\begin{equation}
  \EE^{i}(M_{i,i-1})
  \stackrel{\delta_{i}}\longrightarrow \EE^{i+1}(M_{i+1,i}).
\end{equation}

\begin{lemma}
  There is an isomorphism of \(R\)-modules, natural in~\(X\),
  \begin{equation*}
    H^{i}(\AB^{*}(X)) \cong \ker\delta_{i} \bigm/ \im\delta_{i-1}.
  \end{equation*}
\end{lemma}

\begin{proof}
  By Proposition~\ref{thm:Ext-i-j-0}, the universal coefficient spectral sequence
  \begin{equation*}
    E_{2}^{p} = \Ext_{R}^{p}(\hHT_{*}(X_{i+1},X_{i-1}),R) \;\Rightarrow\; \HT^{*}(X_{i+1},X_{i-1})
  \end{equation*}
  collapses (since \(E_{2}^{p}=0\) unless \(p=i\) or~\(i+1\)),
  and there is a short exact sequence
  \begin{equation}
    \label{eq:Ei-short-exact}
    0 \longrightarrow \EE^{i+1}(M_{i+1,i-1})
    \longrightarrow \HT^{*}(X_{i+1},X_{i-1})
    \longrightarrow \EE^{i}(M_{i+1,i-1})
    \longrightarrow 0
  \end{equation}
  coming from the filtration of the spectral sequence.

  Consider the following (possibly non-commuting) diagram:
  \begin{equation*}
    \begin{tikzcd}[font=\small,column sep=small]
      & \HT^{*}(X_{i+1},X_{i-1}) \arrow{r} \arrow{d} & \HT^{*}(X_{i},X_{i-1}) \arrow{r}{d_{i}} \arrow{d}{\phi_{i}} & \HT^{*}(X_{i+1},X_{i}) \arrow{r} & \HT^{*}(X_{i+1},X_{i-1}) \\
      0 \arrow{r} & \EE^{i}(M_{i+1,i-1}) \arrow{r} & \EE^{i}(M_{i,i-1}) \arrow{r}{\delta_{i}} & \EE^{i+1}(M_{i+1,i}) \arrow{r} \arrow{u}[right]{\psi_{i+1}} & \EE^{i+1}(M_{i+1,i-1}) \arrow{r} \arrow{u} & 0
    \end{tikzcd}
  \end{equation*}
  The rows are part of long exact sequences; the bottom one 
  is based on the bottom row of~\eqref{eq:hHT-Xi-Xi1-diagram} and uses again Proposition~\ref{thm:Ext-i-j-0}.
  The vertical maps come from~\eqref{eq:Ei-short-exact},
  and \(\phi_{i}\)~and~\(\psi_{i+1}\) are isomorphisms, once again by Proposition~\ref{thm:Ext-i-j-0}.
  The left square and the right square commute by naturality.
  Hence \(\phi_{i}\) maps \(\ker d_{i}\) isomorphically onto~\(\ker\delta_{i}\),
  and \(\psi_{i+1}\) maps \(\im\delta_{i}\) isomorphically onto~\(\im d_{i}\).
  
  The maps~\(\phi_{i}\) and \(\psi_{i}\) are induced by the filtration of~\(\HT^{*}(X_{i},X_{i-1})\)
  coming from the universal coefficient spectral sequence. But since \(\Ext_{R}^{j}(\HT^{*}(X_{i},X_{i-1}),R)=0\)
  for \(j\ne i\), the filtration of~\(\HT^{*}(X_{i},X_{i-1})\) has only one non-trivial step, \ie,
  it looks like
  \begin{equation}
    0 = \cdots = 0 = \FF^{i+1} \subset \FF^{i} = \HT^{*}(X_{i},X_{i-1}) = \FF^{i-1} = \cdots = \FF^{0}.
  \end{equation}
  By the properties of spectral sequences the composition~\(\psi_{i}\phi_{i}\)
  is the inclusion~\(\FF^{i}\hookrightarrow\FF^{i-1}\), which in our case is the identity.
  So \(\phi_{i}=\psi_{i}^{-1}\) for any~\(i\).
  
  As a consequence, \(\phi_{i}\) induces an isomorphism
  \begin{equation*}
    \ker d_{i} \bigm/ \im d_{i-1} \to \ker\delta_{i} \bigm/ \im\delta_{i-1}.
    \qedhere
  \end{equation*}
\end{proof}

\begin{proof}[Proof of Corollary~\ref{thm:4:ext-hab}]
  Applying \(\Ext_{R}(-,R)\) to the diagram~\eqref{eq:hHT-Xi-Xi1-diagram}
  leads to the commutative diagram
  \begin{equation*}
    \begin{tikzcd}
      \EE^{i}(\hHT_{*}(X_{i+1})) \arrow{r} & \EE^{i}(\hHT_{*}(X_{i})) \arrow{r} & \EE^{i+1}(M_{i+1,i}) \\ 
      & \EE^{i}(M_{i,i-1}) \arrow{r}{\delta_{i}} \arrow{u} & \EE^{i+1}(M_{i+1,i})\mathrlap{.} \arrow{u}[right]{=}
    \end{tikzcd}
  \end{equation*}
  Together with the analogous square for~\(i-1\) instead of~\(i\)
  we can form the commutative diagram
  \begin{equation*}
    \begin{tikzcd}
      0 \\
      \EE^{i-1}(\hHT_{*}(X_{i-1})) \arrow{r}{=} \arrow{u} & \EE^{i-1}(\hHT_{*}(X_{i-1})) \arrow{d} \\
      \EE^{i-1}(M_{i-1,i-2}) \arrow{r}{\delta_{i-1}} \arrow{u} & \EE^{i}(M_{i,i-1}) \arrow{r}{\delta_{i}} \arrow{d}{p_{i}} & \EE^{i+1}(M_{i+1,i}) \\
      & \EE^{i}(\hHT_{*}(X_{i})) \arrow{r}{=} \arrow{d} & \EE^{i}(\hHT_{*}(X_{i})) \arrow{u} \\
      & 0 & \EE^{i}(\hHT_{*}(X_{i+1})) \arrow{u} \\
      & & 0\mathrlap{,} \arrow{u}
    \end{tikzcd}
  \end{equation*}
  where all columns come from the long exact sequence for~\(\Ext\),
  applied to some row of~\eqref{eq:hHT-Xi-Xi1-diagram}.
  We have used Proposition~\ref{thm:Ext-i-j-0} to obtain the zero entries.
  As a consequence,
  \begin{align}
    \ker\delta_{i} &= p_{i}^{-1}(\EE^{i}(\hHT_{*}(X_{i+1}))) \\
    \intertext{and}
    \im\delta_{i-1} &= \ker p_{i}.
  \end{align}
  Hence
  \begin{equation}
    H^{i}(\AB^{*}(X)) = \ker\delta_{i} \bigm/ \im\delta_{i-1}
    \cong \EE^{i}(\hHT_{*}(X_{i+1}))
    \cong \EE^{i}(\hHT_{*}(X)).
  \end{equation}
  The last isomorphism follows from Proposition~\ref{thm:Ext-i-j-0} and the short sequence
  \begin{multline}
    0 = \EE^{i}(\hHT_{*}(X,X_{i+1})) \to \EE^{i}(\hHT_{*}(X)) \\
    \to \EE^{i}(\hHT_{*}(X_{i+1})) \to \EE^{i+1}(\hHT_{*}(X,X_{i+1})) = 0.
  \end{multline}
  This completes the proof.
\end{proof}

\end{document}